\documentclass%[11pt,letterpaper,reqno]
 {amsart}

\usepackage{amsmath,cite,amsfonts,amssymb,amsthm,amscd,latexsym}

\usepackage{array}
\usepackage{mathrsfs}
\usepackage{color}
 \usepackage{lmodern}
\usepackage[T1]{fontenc}
\usepackage{a4wide}
\usepackage{enumitem}
\setlist[enumerate,1]{label={\roman*)}}

\newtheorem{thm}{Theorem}[section]
\newtheorem{theorem}[thm]{Theorem}
\newtheorem{conjecture}[thm]{Conjecture}

\newtheorem{lemma}[thm]{Lemma}

\newtheorem{corollary}[thm]{Corollary}
\newtheorem{proposition}[thm]{Proposition}

\newtheorem{definition}[thm]{Definition}
\newtheorem{remark}[thm]{Remark}

\begin{document}

\title{On simple transposed Poisson algebras}

\author[Amir Fernández Ouaridi]{Amir Fernández Ouaridi\\
University of Sevilla\\
\texttt{amir.fernandez.ouaridi@gmail.com}} % delete this line if you do not have one yet
\address{Amir Fernández Ouaridi. \newline \indent University of Sevilla, Department of Geometry and Topology, Sevilla (Spain).}
\email{{\tt amir.fernandez.ouaridi@gmail.com}}

\begin{abstract}
We develop a structure theory for transposed Poisson algebras over fields of characteristic different from two. In particular, we prove that every finite-dimensional transposed Poisson algebra over an algebraically closed field decomposes as the direct sum of a unital ideal and a nilpotent ideal. 
As a consequence, we obtain restrictions on simple transposed Poisson algebras and use them to classify the simple finite-dimensional transposed Poisson algebras over an algebraically closed field of characteristic $p>3$. Precisely, we show that every such algebra has as underlying Lie algebra a Zassenhaus algebra $\mathcal{W}(1;n)$ and is isomorphic to one of the algebras of the family $\mathcal{W}_n(q)$ arising from a mutation of a natural associative commutative structure on $\mathcal{W}(1;n)$. 
We then study the corresponding isomorphism problem for the family $\mathcal{W}_n(q)$ and determine the irreducible finite-dimensional representations of these simple transposed Poisson algebras $\mathcal{W}_n(q)$ in the unital case. We conclude with some applications to Jordan superalgebras, weak-Leibniz algebras and quasi-Poisson algebras.
\end{abstract}


%\thanks{ }
\maketitle

\section{Introduction}

\medskip

Transposed Poisson algebras were introduced as a dual of the class of Poisson algebras in which the roles of the two multiplications in the Leibniz rule are exchanged in \cite{tpa}. More precisely, a transposed Poisson algebra is a vector space $\mathcal{P}$ over a field $\mathbb{F}$ of characteristic $p\neq 2$ endowed with an associative commutative multiplication $\circ$ and an associated Lie bracket $[\cdot, \cdot]$ satisfying the transposed Leibniz rule, that is, for any $x,y,z\in \mathcal{P}$ we have:
\begin{equation*}%\label{tpaid}
    2 x \circ [y, z] = [x \circ y, z] + [y, x \circ z]. 
\end{equation*}

The recent interest in this class of algebras is justified by several reasons. The variety of transposed Poisson algebras coincides  with 
the variety of quasi-Poisson algebras, a class with which all the known simple superconformal Lie algebras can be constructed, as Billig has shown in \cite{billig}, when the P map is fixed as the identity. Equivalently, it coincides 
with the variety of commutative Gelfand-Dorfman algebras, which are used as a source of Hamiltonian operators \cite{sgd}. Transposed Poisson algebras are Jordan brackets, which are fundamental in the theory of Jordan algebras \cite{simple}; there is a correspondence with weak-Leibniz algebras by polarization and depolarization, a class that generalizes Leibniz algebras \cite{askar}.  They are also related to Novikov-Poisson algebras, $n$-Lie algebras, pre-Lie algebras, hom-Lie algebras, among other classes, see \cite{chenbai24}. 
Transposed Poisson algebras share some properties with Poisson algebras, including the closure under tensor products and the Koszul self-duality as an operad \cite{tpa}. Also, both classes are F-manifolds \cite{sgd}. At the same time, their intersection as varieties is small \cite{tpa}.

The notion of a Poisson algebra can be motivated by the classical construction in physics of a Lie bracket given an associative commutative algebra with commuting derivations. An analogous example appears in transposed Poisson algebras when we have an associative commutative algebra with a derivation $(\mathcal{A}, \circ, d)$, by defining the Lie bracket $[x,y]=x\circ d(y)- d(x)\circ y$ for any $x, y \in \mathcal{A}$. In this case, $(\mathcal{A}, \circ, [\cdot, \cdot])$ is a transposed Poisson algebra.
An example of a transposed Poisson algebra was constructed on the Witt algebra in \cite{FKL}. The procedure used to construct this example is based on the following fact: the transposed Leibniz rule can be reformulated by saying that each left multiplication operator of the associative commutative algebra is a $\frac{1}{2}$-derivation of the Lie bracket. This realization is fundamental in the description of the transposed Poisson structures on a given Lie algebra, for example, for Witt-type algebras see \cite{kk23} or for twisted Block-Lie superalgebras see \cite{HGS25}, and will also play an important role in the present paper. We refer to the mentioned papers for the precise procedure to construct the transposed Poisson structures on a given Lie algebra using its $\frac{1}{2}$-derivations. Recall that a $\frac{1}{2}$-derivation of a non-associative algebra $(\mathcal{A}, \circ)$ is a linear map $D$ on $\mathcal{A}$ such that $2 D(x\circ y) = D(x)\circ y + x\circ D(y)$. The space of $\frac{1}{2}$-derivations will be denoted by  $\textrm{Der}_{\frac{1}{2}}(\mathcal{A})$. These derivations were well-studied for Lie algebras  by Filippov (see \cite{fili98, fili99}) and Leger and Luks (see \cite{leger}). They appear in other contexts, for example, they were used for the construction of non-semigroup gradings of Lie algebras \cite{z10}. 

One of the main problems in the theory of non-associative algebras is the classification of the simple algebras of a given variety. An ideal of a transposed Poisson algebra is a subspace which is an ideal for both multiplications simultaneously. A simple transposed Poisson algebra is an algebra with only two ideals, zero and the whole algebra. 
In a recent work, it was proved that a simple transposed Poisson algebra has a simple associated Lie bracket \cite[Theorem~11]{simple}. This result implies that any simple finite-dimensional transposed Poisson algebra over an algebraically closed field of characteristic zero is trivial (one of the multiplications is zero), as simple Lie algebras do not admit non-trivial $\frac{1}{2}$-derivations, see \cite[Theorem 6]{fili98}.
In prime characteristic $p>2$, the situation is different. In fact, the author constructs a non-trivial simple finite-dimensional transposed Poisson algebra on the modular Witt algebra, that is, the Lie algebra of derivations of the algebra of truncated polynomials $\mathbb{F}[x]/(x^p)$. This example motivates the central problem addressed in this manuscript: classifying the transposed Poisson algebras over an algebraically closed field of characteristic $p>3$.

Any non-trivial simple transposed Poisson algebra is a transposed Poisson structure on a simple Lie algebra.
In characteristic $p>3$, the simple finite-dimensional Lie algebras over an algebraically closed field are known to be of classical type, of Cartan type, or of Melikyan type (see, for example, \cite{s04}). Moreover, several results in the literature suggest that the simple Lie algebras which may admit non-trivial transposed Poisson structures are extremely restricted.
Indeed, if in a perfect Lie algebra the centroid coincides with the space of $\frac12$-derivations, then every transposed Poisson structure on this Lie algebra is trivial \cite{simple}. On the other hand, Leger and Luks proved that the space of quasiderivations of a simple finite-dimensional Lie algebra of rank greater than one coincides with the direct sum of the derivations and the centroid \cite{leger}. In particular, for a simple finite-dimensional Lie algebra of rank greater than one, the centroid coincides with the space of $\frac12$-derivations. Thus one is led to consider the rank one case.

In characteristic $p>3$, Benkart and Osborn proved that the simple finite-dimensional Lie algebras of rank one are $\mathfrak{sl}_2$ and the Albert--Zassenhaus algebras \cite[Theorem~1.2]{benkart}. The latter belong to the families $\mathcal W(1;n)$ and $\mathcal H(2;\underline{m};\Phi(1))$. Furthermore, Filippov proved that in a prime Lie algebra admitting a non-degenerate symmetric invariant bilinear form, the centroid coincides with the space of $\frac12$-derivations \cite[Theorem~6]{fili98}. Since the Hamiltonian algebras $\mathcal H(2r;\underline{m})$ admit such a form \cite[Section 4.6, Theorem~6.5]{sf88}, this excludes the Hamiltonian case. The same argument also rules out $\mathfrak{sl}_2$ for $p>3$. Therefore, these results indicate that the only family which may carry non-trivial simple transposed Poisson structures is $\mathcal W(1;n)$.

Another useful viewpoint is provided by commutative $2$-cocycles \cite{AP10}. Indeed, we can consider the symmetric bilinear map
$\varphi(x,y)=x\circ y$. Then the transposed Leibniz identity implies that $\varphi$ induces a non-zero commutative $2$-cocycle on the Lie algebra $(\mathcal P,[\cdot,\cdot])$ after composing with any $\lambda\in \mathcal{P}^*$ such that $\lambda(\mathcal{P} \circ \mathcal{P})\neq 0$. Therefore, non-trivial simple transposed Poisson algebras are closely related to simple Lie algebras admitting non-zero commutative $2$-cocycles.

The paper is organized as follows. In Section~\ref{sec2}, we develop a decomposition theory for finite-dimensional transposed Poisson algebras over algebraically closed fields. In Section~\ref{sec3}, we study simple transposed Poisson algebras in prime characteristic. In particular, in characteristic $p>3$ we reduce the problem to transposed Poisson algebras whose underlying Lie algebra is a Zassenhaus algebra, and show that every simple finite-dimensional non-trivial transposed Poisson algebra is isomorphic to one of the algebras $\mathcal W_n(q)$. In Section~\ref{sec4}, we study the isomorphism problem for the family $\mathcal W_n(q)$ and obtain normalized forms for the parameter $q$. In Section~\ref{sec5}, we determine the irreducible finite-dimensional representations in the unital case. Finally, in Section~\ref{sec6}, we discuss applications to Jordan superalgebras, weak-Leibniz algebras, and quasi-Poisson algebras.

\section{A decomposition of transposed Poisson algebras}
\label{sec2}

Unless stated otherwise, all results in this paper hold for a transposed Poisson algebra of arbitrary dimension and over an arbitrary field of characteristic different from two. The associative commutative multiplication of any transposed Poisson algebra is not assumed to be unital, and it will be denoted by concatenation when the context is clear.
The purpose of this section is to obtain a useful decomposition of finite-dimensional transposed Poisson algebras over algebraically closed fields. 
Let $(\mathcal{P}, \circ, [\cdot, \cdot])$ be a transposed Poisson algebra. Denote by $P_x$ and $Q_x$ the maps in $End(\mathcal{P})$ given by $P_x(y) = xy$ and $Q_x(y) = [x, y]$ for $x, y \in \mathcal{P}$. A straightforward manipulation of the identities of transposed Poisson algebras gives rise to several relations between these two operators, see \cite{tpa, simple}. The next lemmas include some of them.

\begin{lemma}\label{lemid4}
Let $(\mathcal{P}, \circ, [\cdot, \cdot])$ be a transposed Poisson algebra. Then for $u, v, x, y \in \mathcal{P}$ we have
    \begin{equation}
        [ux, vy] = uv[x, y] + xy[u, v].   \end{equation} 
\end{lemma}
\begin{proof}
By \cite[Theorem 2.7 (13)]{tpa}, we have 
$$2uv[x, y] = [ux, vy] + [vx, uy], \quad \quad 2xy[u, v] = [ux, vy] + [uy, vx]$$
Summing both equations, we obtain the identity in the lemma.
\end{proof}

\begin{lemma}\label{rels1}
    Let $(\mathcal{P}, \circ, [\cdot, \cdot])$ be a transposed Poisson algebra. Then the following relations hold for all $x, y, z \in \mathcal{P}$ and $k\in \mathbb{N}$.
    \begin{enumerate}
        \item $P_x^{2k} Q_y = Q_{x^{k}y}P_x^k$,
        \item  $P_x^{2k-1} Q_y = \frac{1}{2}Q_{x^{k-1}y}P_x^k + \frac{1}{2}Q_{x^{k}y}P_x^{k-1}$,
        \item $2P_{xy}Q_z = Q_{xz}P_y + Q_{yz}P_x$.
    \end{enumerate} 
\end{lemma}
\begin{proof}
The relations are a consequence of \cite[Theorem 2.7 (13)]{tpa}. 
\end{proof}

Using these lemmas, we establish the following key relations.

\begin{lemma}\label{rels2}
Let $(\mathcal{P}, \circ, [\cdot, \cdot])$ be a transposed Poisson algebra. Then for any $x, y \in \mathcal{P}$ and $k, s\in \mathbb{N}$ with $s\leq k$, we have the following relation.  
\begin{enumerate} 
    \item If $s$ is even, then
\begin{equation}\label{eq11}
    \binom{2k}{s} Q_{x^{k-\frac{s}{2}}y} P_x^{k-\frac{s}{2}} - \sum_{j=0}^s \binom{k}{j}\binom{k}{s-j} Q_{x^{k-j}y} P_x^{k-s+j} = 0.
\end{equation}
    
    \item If $s$ is odd, then
    \begin{equation}\label{eq12}
\binom{2k}{s} Q_{x^{k - \frac{s+1}{2}}y}P_x^{k - \frac{s-1}{2}}  +   \binom{2k}{s} Q_{x^{k - \frac{s-1}{2}}y}P_x^{k - \frac{s+1}{2}} - 2\sum_{j=0}^s \binom{k}{j} \binom{k}{s-j} Q_{x^{k-j}y} P_x^{k-s+j}=0.
\end{equation}
\end{enumerate}
\end{lemma}
\begin{proof}
    Using the Vandermonde identity and the relation ($iii$) in Lemma~\ref{rels1}, we obtain
\begin{equation*}
    \begin{split}
    \textrm{LHS (\ref{eq11})}=&\sum_{j=0, j\neq \frac{s}{2}}^s \binom{k}{j}\binom{k}{s-j} P_x^{2k-s} Q_{y}   - \sum_{j=0, j\neq \frac{s}{2}}^s \binom{k}{j}\binom{k}{s-j} Q_{x^{k-j}y} P_x^{k-s+j}\\
    =&\sum_{j=0, j\neq \frac{s}{2}}^s \binom{k}{j}\binom{k}{s-j} P_x^{2k-s} Q_{y}   - 2 \sum_{j=0}^{\frac{s}{2}-1} \binom{k}{j}\binom{k}{s-j}  P_x^{2k-s} Q_{y} = 0.
    \end{split}
\end{equation*}   
    That proves the first identity.  Likewise, for the second relation we have
    \begin{equation*}
    \begin{split}
    \textrm{LHS (\ref{eq12})}=& \sum_{j=0, j\neq \frac{s\pm 
1}{2}}^s \binom{k}{j}\binom{k}{s-j} 
 Q_{x^{k - \frac{s+1}{2}}y}P_x^{k - \frac{s-1}{2}}  + 
 \sum_{j=0, j\neq \frac{s\pm 
1}{2}}^s \binom{k}{j}\binom{k}{s-j}  Q_{x^{k - \frac{s-1}{2}}y}P_x^{k - \frac{s+1}{2}} \\
    &- 2\sum_{j=0, j\neq \frac{s\pm 
1}{2}}^s \binom{k}{j} \binom{k}{s-j} Q_{x^{k-j}y} P_x^{k-s+j}\\
= \, \, \, &  2\sum_{j=0, j\neq \frac{s\pm 
1}{2}}^s \binom{k}{j}\binom{k}{s-j} 
 P_x^{2k - s}Q_{y} - 4\sum_{j=0}^{\frac{s-1}{2}-1} \binom{k}{j} \binom{k}{s-j} P_x^{2k-s}Q_{y}  =0.    \end{split}\end{equation*}
\end{proof}

\begin{remark}
Let $(\mathcal{P}, \circ, [\cdot, \cdot])$ be a transposed Poisson algebra  over an arbitrary field $\mathbb{F}$. 
    It can be checked, using the identities above, that for any $x\in \mathcal{P}$, we have that $ker(P_x)$, $img(P_x)$ and $ker(P_x) + img(P_x)$ are subalgebras of $\mathcal{P}$. 
\end{remark}

Moreover,
it is a consequence of the transposed Leibniz rule that the eigenspaces of the operators $P_a$, i.e., the spaces $\mathcal{P}^1_{a,\lambda}= \left\{x\in\mathcal{P}: P_{a} x=\lambda x \right\}$ where $\lambda \in \mathbb{F}$, are  subalgebras of $\mathcal{P}$ too. A more general fact holds for the generalized eigenspaces as the next result shows.

\begin{proposition}\label{manyideals}
    Let $(\mathcal{P}, \circ, [\cdot, \cdot])$ be a transposed Poisson algebra over an arbitrary field $\mathbb{F}$. Then the generalized eigenspace $$\mathcal{P}_{a,\lambda}= \left\{x\in\mathcal{P}: (P_{a}-\lambda)^{k}x=0 \textrm{ for some $k>0$}\right\}$$
    of the operator $P_a$ is an ideal of the transposed Poisson algebra $\mathcal{P}$ for every $a\in \mathcal{P}$ and $\lambda\in \mathbb{F}$.
\end{proposition}
\begin{proof}
    Fix $a\in \mathcal{P}$ and $\lambda\in \mathbb{F}$. Let $x\in\mathcal{P}_{a,\lambda}$ and $y\in \mathcal{P}$. Denote by $m$ the minimal number such that $(P_{a}-\lambda)^{m}x = 0$. By the commutativity and the associativity of $(\mathcal{P}, \circ)$, we have $[P_a, P_y] = 0$. Thus, we obtain $(P_{a}-\lambda)^{k}P_{y}x=  0$, for $k\geq m$. Now, let us show that $(P_{a}-\lambda)^{k}Q_{y}x= 0$, for a suitable $k\geq 2m$. 
First, we use the relations (1) and (2) in Lemma~\ref{rels1} to expand the expression.
\begin{equation*}%\label{eq1}
    \begin{split}
        (P_a-\lambda)^{2k}Q_yx =& \sum_{i=0}^{2k} (-1)^{i} \binom{2k}{i}\lambda^{i} P_a^{2k-i}Q_y x \\      =&\sum_{i=0}^{k} \binom{2k}{2i}\lambda^{2i} P_a^{2(k-i)}Q_y x - 
\sum_{i=0}^{k-1} \binom{2k}{2i+1}\lambda^{2i+1} P_a^{2(k - i)-1}Q_y x \\
        =&\sum_{i=0}^{k} \binom{2k}{2i}\lambda^{2i} Q_{a^{k-i}y} P_a^{k-i} x - \frac{1}{2}
\sum_{i=0}^{k-1} \binom{2k}{2i+1}\lambda^{2i+1} Q_{a^{k - i-1}y}P_a^{k - i} x \\
& \quad\quad\quad\quad\quad\quad\quad\quad\quad\quad\quad\,\, - 
\frac{1}{2}\sum_{i=0}^{k-1}  \binom{2k}{2i+1}\lambda^{2i+1} Q_{a^{k - i}y}P_a^{k - i-1} x.
    \end{split}
\end{equation*}

We sum the following terms for $0\leq j\leq k$, which are equal to zero, to the expression above. 
$$0=(-1)^{j-1}\binom{k}{j}\lambda^jQ_{a^{k-j}y}(P_{a}-\lambda)^{k}x = \binom{k}{j}\sum_{i=0}^{k} (-1)^{i+j-1} \binom{k}{i}\lambda^{i+j} Q_{a^{k-j}y} P_a^{k-i} x.$$

The resulting expression is
\begin{equation*}%\label{eq2}
        \begin{split}
        &\sum_{i=0}^{k} \binom{2k}{2i}\lambda^{2i} Q_{a^{k-i}y} P_a^{k-i} x - \frac{1}{2}
\sum_{i=0}^{k-1} \binom{2k}{2i+1}\lambda^{2i+1} Q_{a^{k - i-1}y}P_a^{k - i} x \\
&- 
\frac{1}{2}\sum_{i=0}^{k-1}  \binom{2k}{2i+1}\lambda^{2i+1} Q_{a^{k - i}y}P_a^{k - i-1} x + \sum_{j=0}^k \binom{k}{j}\sum_{i=0}^{k} (-1)^{i+j-1} \binom{k}{i}\lambda^{i+j} Q_{a^{k-j}y} P_a^{k-i} x.
    \end{split}
\end{equation*}

Now, we obtain the desired result by using the relations in Lemma~\ref{rels2} for all the terms with a fixed $\lambda^s$ for $0\leq s\leq 2k$. Indeed, we have one of the following situations.

\begin{enumerate}
    \item If $s$ is even, then the terms with $\lambda^s$ are precisely
    $$\binom{2k}{s}\lambda^{s} Q_{a^{k-\frac{s}{2}}y} P_a^{k-\frac{s}{2}} x - \sum_{j=0}^s \binom{k}{j}\binom{k}{s-j}\lambda^{s} Q_{a^{k-j}y} P_a^{k-s+j} x = 0.$$
    
    \item If $s$ is odd, then we have the terms

\begin{equation*}%\label{eq2}
        \begin{split}
         - \frac{1}{2}
\binom{2k}{s}\lambda^{s} Q_{a^{k - \frac{s-1}{2}-1}y}P_a^{k - \frac{s-1}{2}} x & - 
\frac{1}{2} \binom{2k}{s}\lambda^{s} Q_{a^{k - \frac{s-1}{2}}y}P_a^{k - \frac{s-1}{2}-1} x  \\
& +\sum_{j=0}^s \binom{k}{j} \binom{k}{s-j}\lambda^{s} Q_{a^{k-j}y} P_a^{k-s+j} x = 0
    \end{split}
\end{equation*}

\end{enumerate}

Since for every $\lambda^s$ we have shown that the corresponding terms add up to zero, we obtain that $(P_a-\lambda)^{2k}Q_yx = 0$ for $k\geq 2m$. Therefore, we have proved that $P_yx, Q_yx \in \mathcal{P}_{a,\lambda}$ for all $x\in\mathcal{P}_{a,\lambda}$ and $y\in \mathcal{P}$. Hence, we conclude that the vector space $\mathcal{P}_{a,\lambda}$ is an ideal of $\mathcal{P}$ for $a\in \mathcal{P}$ and $\lambda\in \mathbb{F}$.
\end{proof}

\begin{corollary}\label{simpleopers}
Let $(\mathcal{P}, \circ, [\cdot, \cdot])$ be a simple transposed Poisson algebra over an arbitrary field $\mathbb{F}$. Then $\mathcal{P}_{x,\lambda}=0$ or $\mathcal{P}_{x,\lambda}=\mathcal{P}$  
    for every $x\in \mathcal{P}$ and $\lambda\in \mathbb{F}$.
\end{corollary}

We prove that a finite-dimensional transposed Poisson algebra can be directly decomposed in two main ideals, one of them being unital, as the following theorem shows. We use the symbol $\oplus$ for the direct sum of transposed Poisson algebras. First, we need this auxiliary lemma.

\begin{lemma}\label{lemuni}
Let $(\mathcal{A},\circ)$ be an associative commutative algebra of dimension $n \in \mathbb{N}$. If for every $x\in \mathcal{A}$,  the characteristic polynomial of $P_x$ is $\chi_x(t) = (t-\lambda_x)^n$. Then $(\mathcal{A},\circ)$ is nilpotent or unital.
\end{lemma}

\begin{proof}
Suppose $(\mathcal{A},\circ)$ is not nilpotent, then it has an idempotent element $e\in \mathcal{A}$. Thus, for every $x\in \mathcal{A}$ we have 
$
e\circ (e \circ x -x)= (e \circ e) \circ x - e\circ x = 0$. 
Since $\chi_e(t) = (t-\lambda_e)^n$ and $\lambda_e = 1$, we have that $P_e$ is invertible. Hence, we deduce $e\circ x = x$ and $e$ is the unit.
\end{proof}

\begin{theorem}\label{decomp}
Let $(\mathcal{P},\circ,[\cdot,\cdot])$ be a finite-dimensional transposed Poisson algebra over an
algebraically closed field $\mathbb F$. Then there exist ideals $\mathcal D,\mathcal N$ of $\mathcal P$
such that
$$
\mathcal P=\mathcal D\oplus \mathcal N,\qquad
(\mathcal D,\circ)\ \text{is unital},\qquad
(\mathcal N,\circ)\ \text{is nilpotent}.
$$
Moreover, $\mathcal D$ can be decomposed as
$$
\mathcal D=\bigoplus_{i\in\Delta}\mathcal D_i
$$
where each $\mathcal D_i$ is an ideal and, for every $x\in \mathcal D_i$,
the restriction $P_x|_{\mathcal D_i}$ has a unique eigenvalue.
\end{theorem}
\begin{proof}
We proceed by induction on $n=\textrm{dim}(\mathcal P)$. The case $n=1$ is clear.
Assume the result holds for transposed Poisson algebras of dimension at most $k$, and let
$\textrm{dim}(\mathcal P)=k+1$.

If for every $x\in \mathcal{P}$ we have that $P_x$ has a unique eigenvalue $\lambda_x$, then we are done by Lemma~\ref{lemuni}.  
Otherwise, choose $x\in \mathcal{P}$ such that $P_x$ has at least two distinct eigenvalues. Write $\Sigma_x$ for the spectrum of $P_x$. By
Proposition~\ref{manyideals}, we have
$$
\mathcal P=\bigoplus_{\lambda\in\Sigma_x}\mathcal P_{x,\lambda},
$$
where each $\mathcal P_{x,\lambda}$ is an ideal of $\mathcal P$. Since each $\mathcal P_{x,\lambda}$ has dimension at most $k$, for every
$\lambda\in\Sigma_x$ we have $
\mathcal P_{x,\lambda}=\mathcal D_\lambda\oplus \mathcal N_\lambda,$ 
where $(\mathcal D_\lambda,\circ)$ is
a unital ideal and $(\mathcal N_\lambda,\circ)$ is a nilpotent ideal, and
$$
\mathcal D_\lambda=\bigoplus_{j\in\Delta_\lambda}\mathcal D_{\lambda,j},$$
with each $\mathcal D_{\lambda,j}$ an ideal such that $P_y|_{\mathcal D_{\lambda,j}}$ has a unique
eigenvalue for $y\in\mathcal D_{\lambda,j}$. Set 
$\mathcal D=\oplus_{\lambda\in\Sigma_x}\mathcal D_\lambda,$ and 
$\mathcal N=\oplus_{\lambda\in\Sigma_x}\mathcal N_\lambda.$
Then $\mathcal D$ and $\mathcal N$ are ideals and
$\mathcal P=\mathcal D\oplus\mathcal N.$
In fact, $(\mathcal D,\circ)$ is
unital, with unit equal to the sum of the units of $\mathcal D_\lambda$, and $(\mathcal N,\circ)$
is nilpotent.
Also, $\mathcal{D}$ has the required property.
\end{proof}

We emphasize that this decomposition implies that $[\mathcal{D}_i, \mathcal{D}_j] = 0$, $\mathcal{D}_i\mathcal{D}_j = 0$, $[\mathcal{D}, \mathcal{N}] = 0$ and $\mathcal{D} \mathcal{N} = 0$ for $i\neq j$. Observe that the component $\mathcal{D}$ is well understood, since it is equivalent to an associative commutative algebra equipped with a derivation, as the Lie bracket is prescribed. Indeed, it was shown in \cite{tpa} that in a unital transposed Poisson algebra $(\mathcal{P}, \circ, [\cdot, \cdot])$, the Lie bracket is determined by $\circ$ and $Q_1$, where $1$ is the unit of $(\mathcal{P}, \circ)$. Precisely, it turns out that $Q_1$ is a derivation of $(\mathcal{P}, \circ)$ and we have $[x, y] = x[1, y] - [1, x] y$ for every $x, y \in \mathcal{P}$.  
By contrast, the component $\mathcal{N}$ has no such a nice description, and is the
only genuinely new part.

\begin{remark}
The finite-dimensional hypothesis in Theorem~\ref{decomp} is required. Indeed, the conclusion may fail even for associative algebras, as shown by the infinite-dimensional algebra $x\mathbb{F}[x]$.
\end{remark}

Moreover, recall that an algebra $\mathcal{P}$ is {\it directly indecomposable} if for any ideals $\mathcal{I}$, $\mathcal{J}$ such that $\mathcal{P} = \mathcal{I} \oplus \mathcal{J}$, we have that $\mathcal{I} = 0$ or $\mathcal{J}=0$. Then the following result is obtained.

\begin{corollary}
    Let $(\mathcal{P}, \circ, [\cdot, \cdot])$ be a transposed Poisson algebra  of dimension $n \in \mathbb{N}$ over an algebraically closed field $\mathbb{F}$. If the Lie algebra $(\mathcal{P}, [\cdot, \cdot])$ is directly indecomposable, then for every $x\in \mathcal{P}$ the characteristic polynomial of $P_x$ is $\chi_x(t) = (t-\lambda_x)^n$. 
\end{corollary}

\begin{remark}\label{rem1}
    In particular, if $(\mathcal{P}, \circ, [\cdot, \cdot])$ is a simple finite-dimensional transposed Poisson algebra over an algebraically closed field $\mathbb{F}$, then $(\mathcal{P}, [\cdot, \cdot])$ is a simple Lie algebra \cite[Theorem 11]{simple}. Hence, for every $x\in \mathcal{P}$ the characteristic polynomial of $P_x$ is $\chi_x(t) = (t-\lambda_x)^n$. Therefore, $(\mathcal{P}, \circ)$ is unital or nilpotent. 
\end{remark}

\section{Simple transposed Poisson algebras in prime characteristic}\label{sec3}

In this section we classify the simple finite-dimensional non-trivial transposed Poisson algebras over an algebraically closed field of characteristic \(p>3\). We first prove, using commutative \(2\)-cocycles, that the underlying Lie algebra is necessarily a Zassenhaus algebra. We then determine all transposed Poisson structures on the Zassenhaus algebras and obtain the classification.

\subsection{The underlying Lie algebra is a Zassenhaus algebra}

Recall that a {\it commutative $2$-cocycle} on a Lie algebra $\mathcal L$ is a symmetric bilinear map
$\varphi:\mathcal L\times \mathcal L\to \mathbb F$ such that for all $x,y,z\in \mathcal L$ we have
$
\varphi([x,y],z)+\varphi([z,x],y)+\varphi([y,z],x)=0.
$

\begin{proposition}\label{mainprop}
Let $(\mathcal{P},\circ,[\cdot,\cdot])$ be a simple finite-dimensional non-trivial transposed
Poisson algebra over an algebraically closed field $\mathbb F$ of characteristic $p>3$.
Then the Lie algebra $(\mathcal{P},[\cdot,\cdot])$ is isomorphic to the Zassenhaus algebra $\mathcal{W}(1;n)$ for some
$n\in\mathbb N$. 
\end{proposition}
\begin{proof}
Denote by $\mathcal L=(\mathcal P,[\cdot,\cdot])$ the associated Lie algebra. Since $\mathcal P$ is simple and non-trivial, the Lie algebra $\mathcal L$ is simple by \cite[Theorem~11]{simple}.
The product $\circ$ is symmetric and, since $\textrm{char}(\mathbb F)\neq 2$, the transposed Leibniz identity implies $
x\circ [y,z]+y\circ [z,x]+z\circ [x,y]=0
$. Choose $\lambda\in \mathcal P^*$ such that $\lambda(\mathcal P\circ \mathcal P)\neq 0$, and define $
\psi(x,y)=\lambda(x\circ y).$ Then $\psi$ is a non-zero commutative $2$-cocycle on the simple Lie algebra $\mathcal L$. Hence, by \cite[Corollary 2.5]{AP10}, the Lie algebra $\mathcal L$ is isomorphic either to $\mathfrak{sl}_2$ or to a Zassenhaus algebra $\mathcal W(1;n)$.

It remains to exclude the case $\mathfrak{sl}_2$. If $\mathcal L\cong \mathfrak{sl}_2$, then $\mathfrak{sl}_2$ admits no non-trivial $\frac12$-derivations in characteristic $p>3$ as it has a nondegenerate symmetric invariant bilinear form, see \cite[Corollary~3]{fili98}. Therefore, every transposed Poisson structure on $\mathfrak{sl}_2$ is trivial, a contradiction. Thus
$(\mathcal P,[\cdot,\cdot])\cong \mathcal W(1;n)
$
for some $n\in \mathbb N$.
\end{proof}

\subsection{The classification}

By Proposition~\ref{mainprop}, to obtain the classification of simple transposed Poisson algebras in prime characteristic we first have to construct all the possible transposed Poisson structures on $\mathcal{W}(1; n)$. 
 Let us recall the definition of the Zassenhaus algebra $\mathcal{W}(1; n)$, see \cite[Chapter 2]{s04} for details. 
Assume the base field $\mathbb{F}$ is algebraically closed and of characteristic $p>2$.
The divided powers algebra $\mathcal{O}(1; n)$ is the $p^n$-dimensional associative commutative algebra with unit defined by the generators $x^{(i)}$ for $0\leq i < N$, with $N=p^n$ and $n\geq 0$, and multiplication 
$$ x^{(i)}x^{(j)} =\binom{i+j}{j}x^{(i+j)},$$
with the convention $x^{(i)}=0$ if $i\geq N$. Also, we may write the unit $x^{(0)}$ as $1$. 
It is isomorphic to the algebra of truncated polynomials $\mathbb{F}[x_1, \ldots, x_n]/(x_1^p, \ldots, x_n^p)$. The Zassenhaus algebra $\mathcal{W}(1; n)$ is its algebra of special derivations, which can be identified with $\mathcal{O}(1; n)\partial$, where $\partial(x^{(i)})=x^{(i-1)}$. For the basis $e_i=x^{(i+1)}\partial$ where $-1\leq i\leq N-2$, the multiplication is given by the bracket
\begin{equation}\label{usual}
    [e_i, e_j]=  \partial(x^{(j+1)}) x^{(i+1)} \partial - \partial(x^{(i+1)}) x^{(j+1)}\partial
    = \left[\binom{i+j+1}{j}-\binom{i+j+1}{i}\right]e_{i+j}.
\end{equation}

Note that the algebra $\mathcal{W}(1; 1)$ is isomorphic to the modular Witt algebra. The transposed Poisson structures on this algebra were determined in \cite[Section 2.3]{simple}.
In order to obtain the transposed Poisson structures on $\mathcal{W}(1; n)$ for $n>1$, we need its $\frac{1}{2}$-derivations. The description of the $\frac{1}{2}$-derivations of $\mathcal{W}(1; n)$ was given  in \cite{z10} using the following realization due to Kuznetsov.  For $n>1$, the algebra $\mathcal{W}(1; n)$ can be seen as a deformation of the current Lie algebra $\mathcal{W}(1; 1)\otimes \mathcal{O}(1; n-1)$ with bracket $\left\{x, y\right\} = [x, y] + \Phi(x, y)$, where $\Phi$ is the $2$-cocycle given by $\Phi(e_i\otimes a, e_j\otimes b) = e_{p-2}\otimes (a\partial(b)-b\partial(a))$ for $i=j=-1$, and zero otherwise.
A basis of the space of $\frac{1}{2}$-derivations of $\mathcal{W}(1; 1)$ is $D^0, D^1, \ldots, D^{p-1}$, where $D$ is given by $D(e_i)=(i+2)e_{i+1}$ assuming the convention $e_k=0$ if $k>p-2$. Using Kuznetsov realization and deformation theory arguments, any $\frac{1}{2}$-derivation of $\mathcal{W}(1; n)$ is then a linear combination of the maps $\varphi \otimes R_u$, where $\varphi$ is a $\frac{1}{2}$-derivation of $\mathcal{W}(1; 1)$ and $R_u$ is the right multiplication operator for $u\in \mathcal{O}(1;n-1)$. 

\medskip

Now, we recall that a {\it mutation}
of an associative commutative algebra $(\mathcal{P}, \circ)$ is a new algebra $(\mathcal{P}, \circ_q)$ where  we have the product $x \circ_q y = x\circ q \circ y$ for any $x, y \in \mathcal{P}$ and a fixed $q \in \mathcal{P}$.  

\begin{definition}
Let $\mathbb F$ be a field of characteristic $p>2$, and let $n\in\mathbb N$.
For $q\in \mathcal W(1;n)$, denote by $\mathcal{W}_n(q)$ the algebra whose Lie bracket is the bracket (\ref{usual}) on $\mathcal W(1;n)$  and whose associative product is the mutation by $q$ of the commutative algebra $(\mathcal W(1;n),\bullet)$, where
\begin{equation}\label{asor}
    e_i\bullet e_j=\binom{i+j+2}{j+1}e_{i+j+1}.
\end{equation}
We may refer to the algebra $\mathcal{W}_n(e_{-1})$ as the model simple transposed Poisson algebra.
\end{definition}

\begin{remark}\label{idntifw}
As we have identified $\mathcal W(1;n)$ with $\mathcal O(1;n)\partial$, the product $(\mathcal W(1;n),\bullet)$ is precisely the product induced by the multiplication in $\mathcal O(1;n)$, namely
$(f\partial)\bullet(g\partial)=(fg)\partial$ for all $f, g\in \mathcal O(1;n).$ Indeed, for $f=x^{(i+1)}$ and $g=x^{(j+1)}$ (the general case follows by bilinearity) we have
$$
(f\partial)\bullet(g\partial)
=
e_i\bullet e_j
=
\binom{i+j+2}{j+1}e_{i+j+1}
=
\binom{i+j+2}{j+1}x^{(i+j+2)}\partial
=
(fg)\partial.
$$
\end{remark}

\begin{lemma}\label{isomm}
   The algebra $(\mathcal{W}(1; n), \bullet)$ is isomorphic to $\mathcal{O}(1;n)$ and $e_{-1}$ is the unital element.
\end{lemma}
\begin{proof}
    The algebra $(\mathcal{W}(1; n), \bullet)$ is isomorphic to $\mathcal{O}(1;n)$ by choosing the homomorphism 
    $$
\Phi:(\mathcal{W}(1;n),\bullet)\longrightarrow \mathcal{O}(1;n),\qquad 
\Phi(e_i)=x^{(i+1)},\quad -1\leq i\leq N-2.
$$
The verification is straightforward. 
\end{proof}

{
\begin{theorem}\label{clasip}
    Let $\mathcal{P}$ be a simple finite-dimensional non-trivial transposed Poisson algebra over an algebraically closed field {of characteristic $p>3$}. Then there is $n\in \mathbb{N}$ such that $\textrm{dim}\,(\mathcal{P}) = p^n$ and  $q\in \mathcal W(1;n)$ such that 
    $\mathcal{P}$ is isomorphic to the algebra $\mathcal{W}_n(q)$. 
\end{theorem}
\begin{proof}        
       Let $\mathcal{P}$ be a simple finite-dimensional non-trivial transposed Poisson algebra over an algebraically closed field $\mathbb{F}$ {of characteristic $p>3$}. By Proposition~\ref{mainprop}, we can assume that the Lie bracket of $\mathcal{P}$ is the Zassenhaus algebra. Therefore, we proceed with the construction of the transposed Poisson structures on $\mathcal{W}(1; n)$. 
       
       A basis of the Zassenhaus algebra $\mathcal{W}(1; n+1)$ with Kuznetsov's realization for $n\geq 0$ is given by $E_{i,j}=e_i\otimes x^{(j)}$ for $-1\leq i \leq p-2$ and $0\leq j \leq p^n-1$. Suppose $(\mathcal{W}(1; 1)\otimes \mathcal{O}(1; n), \circ)$ defines a transposed Poisson structure. Then for every $E_{i,j}$ there is an associated $\frac{1}{2}$-derivation $\phi_{i,j}$, such that $(e_i\otimes x^{(j)})\circ (e_r\otimes x^{(s)}) = \phi_{i,j}(e_r\otimes x^{(s)})$.       
        
        Any $\frac{1}{2}$-derivation is a linear combination of maps $\varphi\otimes R_u$ for $\varphi\in \textrm{Der}_{\frac{1}{2}}(\mathcal{W}(1;1))$ and $u\in \mathcal{O}(1;n)$. Let us simply denote it by $\varphi\otimes u$. Also, let us denote $\circ$ by concatenation. A basis for the space of $\frac{1}{2}$-derivations is given by the maps $D^{k}\otimes x^{(l)}$ where $0 \leq k\leq p-1$ and $0\leq l<p^{n}$.        
        So assume $\phi_{i,j}=\sum_{k,l} \varphi^{i,j}_{k,l} (D^{k}\otimes x^{(l)})$ for some $\varphi^{i,j}_{k,l} \in \mathbb{F}$. 
        
        The commutativity of the multiplication $\circ$ is equivalent to       
        $\phi_{i,j}(e_r\otimes x^{(s)}) = \phi_{r,s}(e_i\otimes x^{(j)})$, so we have the relation
        \begin{equation}\label{eq5}
        \begin{split}
           \phi_{i,j}(e_r\otimes x^{(s)})  &= \sum_{k=0}^{p-1}\sum_{l=0}^{p^n -1} \varphi^{i,j}_{k,l} (D^{k}(e_r)\otimes x^{(l)}x^{(s)}) \\ 
            & = \sum_{k=0}^{p-1}\sum_{l=0}^{p^n -1} \varphi^{r, s}_{k,l} (D^{k}(e_i)\otimes x^{(l)}x^{(j)}) =\phi_{r,s}(e_i\otimes x^{(j)}).
        \end{split}
        \end{equation}

        Observe that $D^{k}(e_r)=\prod_{t=1}^{k}(r+t+1)e_{k+r}$, where we assume the convention $e_{i} = 0$ if $i>p-2$. For $i=-1$ and $j=0$, we have the following two equations from the left hand side and the right hand side of equation (\ref{eq5}), respectively.
\begin{equation*}
\begin{split}
        \sum_{k=0}^{p-1}\sum_{l=0}^{p^n -1} \varphi^{-1,0}_{k,l} (D^{k}(e_r)\otimes x^{(l)}x^{(s)}) & = \sum_{k=0}^{p-1}\sum_{l=0}^{p^n -1} \varphi^{-1,0}_{k,l} \prod_{t=1}^{k}(r+t+1) \binom{l+s}{s}(e_{k+r}\otimes x^{(l+s)})\\
        &=\sum_{k=r}^{p+r-1}\sum_{l=s}^{p^n+s -1} \varphi^{-1,0}_{k-r,l-s} \prod_{t=1}^{k-r}(r+t+1) \binom{l}{s}(e_{k}\otimes x^{(l)}).\\
        \sum_{k=0}^{p-1}\sum_{l=0}^{p^n -1} \varphi^{r, s}_{k,l} (D^{k}(e_{-1})\otimes x^{(l)}) &= \sum_{k=0}^{p-1}\sum_{l=0}^{p^n -1} \varphi^{r, s}_{k,l} k! (e_{k-1}\otimes x^{(l)}) \\ &= \sum_{k=-1}^{p-2}\sum_{l=0}^{p^n -1} \varphi^{r, s}_{k+1,l} (k+1)! (e_{k}\otimes x^{(l)}).
 \end{split}
\end{equation*}

 By equation (\ref{eq5}), we obtain the relations $\varphi^{-1,0}_{k-r,l-s} \binom{l}{s} = \varphi^{r, s}_{k+1,l} (r+1)!$ for $r\leq k\leq p-2$ and $s\leq l\leq p^n-1$, and $\varphi^{r, s}_{k+1,l} = 0$ for $k<r$ or $l<s$. Equivalently, we can write them as
\begin{equation*}
\varphi^{r,s}_{k,l}=
\begin{cases}
\dfrac{\varphi^{-1,0}_{k-r-1,\;l-s}\binom{l}{s}}{(r+1)!},
& \text{if } r+1\leq k\leq p-1 \textrm{ and } s\leq l\leq p^{\,n}-1,\\[6pt]
0,
& \text{if } k\leq r \text{ or } l<s.
\end{cases}
\end{equation*}

These conditions are enough to obtain the next description of the product: 
\begin{equation*} 
\begin{split}
\phi_{r,s}(e_i\otimes x^{(j)}) &= \sum_{k=0}^{p-1}\sum_{l=0}^{p^n -1} \varphi^{r, s}_{k,l} (D^{k}(e_i)\otimes x^{(l)}x^{(j)})\\ & = \sum_{k=r+1}^{p-1}\sum_{l=s}^{p^n -1} \frac{\varphi^{-1,0}_{k-r-1,l-s}}{(r+1)!} \binom{l}{s}  \prod_{t=1}^{k}(i+t+1) (e_{k+i}\otimes x^{(l)} x^{(j)})\\
 & =  \sum_{k=-1}^{p-r-3}\sum_{l=0}^{p^n-s -1} \frac{\varphi^{-1,0}_{k+1,l}}{(r+1)!} \binom{l+s}{s} \prod_{t=1}^{k+r+2}(i+t+1) (e_{k+r+i+2}\otimes x^{(l+s)}x^{(j)})\\
 & = \sum_{k=-1}^{p-2}\sum_{l=0}^{p^n-1}\varphi^{-1,0}_{k+1,l} (k+1)! (e_{k}\cdot e_r\cdot e_i \otimes x^{(l)}x^{(s)}x^{(j)})\\
 &=\left[ \sum_{k=-1}^{p-2}\sum_{l=0}^{p^n -1}\varphi^{-1,0}_{k+1,l} (k+1)! (e_{k}\otimes x^{(l)})\right]*(e_{r}\otimes x^{(s)})*(e_i\otimes x^{(j)}),
 \end{split}
\end{equation*}
where $(\mathcal{W}(1; 1)\otimes \mathcal{O}(1; n), *)$ is the algebra with product $(e_i\otimes x^{(j)})*(e_{r}\otimes x^{(s)}) = e_{i}\cdot e_{r}\otimes x^{(j)} x^{(s)}$ and $(\mathcal{W}(1;1), \cdot)$ is the algebra with product $e_i\cdot e_j= \binom{i+j+2}{j+1}e_{i+j+1}$. 

Consider the linear isomorphism given by $e_i\otimes x^{(j)}\rightarrow e_{jp+i}$ for $-1\leq i\leq p-2$ and $0\leq j < p^n$. It is an isomorphism between Kuznetsov realization and $\mathcal{W}(1; n+1)$ with the usual realization (\ref{usual}).
A straightforward verification shows that it maps $(\mathcal{W}(1; 1)\otimes \mathcal{O}(1; n), *)$ to $(\mathcal{W}(1; n+1), \bullet)$, where the product is given by the rule 
introduced in equation (\ref{asor}). 
Furthermore, setting 
$$q= \sum_{k=-1}^{p-2}\sum_{l=0}^{p^n -1}\varphi^{-1,0}_{k+1,l} (k+1)! (e_k\otimes x^{(l)}),$$ it is clear that any transposed Poisson structure is a mutation of the associative and commutative algebra $(\mathcal{W}(1; 1)\otimes \mathcal{O}(1; n), *)$. 
Finally, being a mutation of an associative commutative algebra already implies the associativity and the commutativity.  
\end{proof}
} 

\subsection{The classification in characteristic three}

The classification obtained in Theorem~\ref{clasip} is stated for characteristic $p>3$. The argument used reduces the problem to the Zassenhaus algebras $\mathcal{W}(1;n)$, and then all the corresponding transposed Poisson structures are described by the family $\mathcal{W}_n(q)$. It is natural to ask whether the same description remains valid for $p=3$.

\begin{conjecture}
Let $\mathbb{F}$ be an algebraically closed field of characteristic three. Then every simple finite-dimensional non-trivial transposed Poisson algebra over $\mathbb{F}$ is isomorphic to $\mathcal{W}_n(q)$ for $n\in \mathbb{N}$ and $q\in \mathcal{W}(1;n)$. Equivalently, the associated Lie algebra of every simple finite-dimensional non-trivial transposed Poisson algebra in characteristic three should be a Zassenhaus algebra $\mathcal{W}(1;n)$.
\end{conjecture}

The difficulty in characteristic three lies in the reduction step on the Lie algebra side. More precisely, the proof in characteristic $p>3$ uses the existence of a subalgebra of codimension one together with the description of simple Lie algebras having such a subalgebra. The latter description still applies in characteristic three. Namely, by \cite[Theorem 3.9]{BIO} simple Lie algebras over an algebraically closed field of characteristic $p>2$ with a subalgebra of codimension one are $\mathfrak{sl}_2$ or a Zassenhaus algebra.

On the other hand, by the decomposition results of Section~\ref{sec2}, if $(\mathcal P,\circ,[\cdot,\cdot])$ is a simple finite-dimensional transposed Poisson algebra over an algebraically closed field, then the associative multiplication is either unital or nilpotent. In the unital case, we can still easily prove the existence of a subalgebra of codimension one in the associated Lie algebra.

\begin{lemma}\label{codim1tpauni}
Let $(\mathcal{P}, \circ, [\cdot, \cdot])$ be a finite-dimensional transposed Poisson algebra  over an algebraically closed field $\mathbb{F}$.
Assume that $(\mathcal P,\circ)$ is unital. Then the Lie algebra $(\mathcal P,[\cdot,\cdot])$
has a subalgebra of codimension one.
\end{lemma}

\begin{proof}
Let $A$ be a maximal ideal of $(\mathcal{P}, \circ)$. Then the quotient $\mathcal{P}/A$ is a simple associative commutative algebra, hence one-dimensional as the base field is algebraically closed.
It follows that $A$ is a Lie subalgebra of $(\mathcal P,[\cdot,\cdot])$, since in a unital transposed Poisson algebra we have $[a, b] = a[1,b]-[1,a]b\in A$, for every $a,b \in A$.
\end{proof}

Therefore, the conjecture is reduced to the nilpotent case. More precisely, it remains to determine whether the Lie algebra of a simple finite-dimensional transposed Poisson algebra with nilpotent associative multiplication must admit a subalgebra of codimension one in characteristic three. This question is particularly delicate because, unlike the case $p>3$, the classification of finite-dimensional simple Lie algebras over an algebraically closed field of characteristic three is not yet complete. However, the computations carried out by the author in GAP for several known families in characteristic three, including examples due to Kostrikin, Brown, Ermolaev, Kuznetsov and Skryabin, do not produce any counterexample.

\section{The family $\mathcal{W}_n(q)$}\label{sec4}

The purpose of this section is to determine when two algebras $\mathcal{W}_n(q)$ and $\mathcal{W}_n(q')$ are isomorphic. Given a non-zero element  $q=\sum_{i=-1}^{p^n-2}q_i e_i\in \mathcal W(1;n)$, we write $\nu(q)=\min\{\,i:\ q_i\neq 0\,\}$. Also, denote by $\tilde{q}$ the element in $\mathcal{O}(1;n)$ such that $q = \tilde{q}\partial$. Precisely, we have $\tilde{q} = \Phi(q)$, where $\Phi$ is the isomorphism defined in Lemma~\ref{isomm}.

\begin{lemma}
An element
$$
a=\sum_{i=-1}^{p^n-2} a_i e_i\in (\mathcal W(1;n),\bullet)
$$
is invertible if and only if $a_{-1}\neq 0$.
\end{lemma}

\begin{proof}
The algebra  $(\mathcal W(1;n),\bullet)$ is local and the maximal ideal $I$ of $(\mathcal W(1;n),\bullet)$ is spanned by $\left\{e_0, \ldots, e_{p^n-2}\right\}$. Therefore, an element $a\in (\mathcal W(1;n),\bullet)$ is invertible if and only if $a\not\in I$. 
\end{proof}

\begin{lemma}\label{invuni}
Given $q\in \mathcal{W}(1;n)$, then $\mathcal{W}_n(q)$ is unital if and only if $q$ is invertible in $(\mathcal W(1;n),\bullet)$. 
\end{lemma}
\begin{proof}
If $\mathcal{W}_n(q)$ is unital with unit $u$, then $e_{-1} = u\bullet_q e_{-1} = q\bullet u = u\bullet q$. Hence, $u$ is the inverse of $q$ in $(\mathcal W(1;n),\bullet)$. 
Conversely, if $q$ is invertible in $(\mathcal W(1;n),\bullet)$, then $q^{-1}$ is a unit. 
\end{proof}

\begin{lemma}\label{lemism}
Let $\mathbb{F}$ be an algebraically closed field of characteristic $p>2$ and let $n\in \mathbb{N}$. Assume $(p,n)\neq (3,1)$. Let $\Psi\in \textrm{Aut}(\mathcal W(1;n))$. Then $\Psi(e_{-1})$ is invertible in the algebra $(\mathcal W(1;n),\bullet)$ and $$ \Psi(x\bullet y)=\Psi(e_{-1})^{-1}\bullet \Psi(x)\bullet \Psi(y) $$ for all $x,y\in \mathcal W(1;n)$.
\end{lemma}
\begin{proof}
    Let $\Psi\in \textrm{Aut}(\mathcal W(1;n))$. If $(p, n)\neq (3, 1)$, then there is an  automorphism $\varphi\in \textrm{Aut}(\mathcal{O}(1;n))$  such that $\Psi(D) = \varphi D \varphi^{-1}$ for $D\in \mathcal W(1;n)$, by     \cite[Theorem 7.3.1]{s04}. Also, we have
    $\Psi(e_{-1}) = h \partial$ where $h = \widetilde{\Psi(e_{-1})}\in \mathcal{O}(1;n)$. Recall $e_{-1} = \partial$. It follows 
    $$h \partial (\varphi(x))  = \Psi(e_{-1})(\varphi(x)) = \varphi e_{-1} \varphi^{-1} (\varphi(x)) = \varphi e_{-1} (x) = \varphi(1) = 1.$$
    Hence,  we deduce $h^{-1} = \partial (\varphi(x))$ in $\mathcal{O}(1;n)$.
    In particular, we have
    $$\Psi(e_{-1})\bullet\partial (\varphi(x)) \partial = h \partial \bullet h^{-1} \partial = e_{-1}.$$
    Therefore, the element $\Psi(e_{-1})$ is invertible in $(\mathcal W(1;n),\bullet)$. 
    
    Similarly, given  $f\in \mathcal{O}(1;n)$ we have $\Psi(f \partial) = f' \partial$ for certain $f'\in \mathcal{O}(1;n)$. Then
    $$f' \partial (\varphi(x)) = \Psi(f \partial) (\varphi(x)) = \varphi (f \partial) \varphi^{-1}(\varphi(x)) = \varphi(f).$$
    So $f' = h \varphi(f)$ and $\Psi(f \partial) = h \varphi(f) \partial$. From here, we obtain 
    $$\Psi(f\partial \bullet g \partial) = \Psi(fg \partial) = h\varphi(fg) \partial = h \varphi(f)\varphi(g) \partial$$
    and  $$\Psi(e_{-1})^{-1}\bullet \Psi(f\partial)\bullet \Psi(g\partial) = h^{-1}\partial\bullet h \varphi(f) \partial\bullet h \varphi(g) \partial = h \varphi(f)\varphi(g) \partial.$$
    Thus $\Psi(f\partial \bullet g \partial) = \Psi(e_{-1})^{-1}\bullet \Psi(f\partial)\bullet \Psi(g\partial)$, as in the statement.   
\end{proof}

\begin{remark}
If $(p,n)=(3,1)$, then $\mathcal W(1;1)\cong \mathfrak{sl}_2$, and the previous result does not hold for every automorphism. Indeed, the bracket is given by
$$
[e_0,e_{-1}]=-e_{-1},\qquad [e_{-1},e_1]=e_0,\qquad [e_0,e_1]=e_1.
$$
The linear map $\Psi:\mathcal W(1;1)\to \mathcal W(1;1)$ defined by
$$
\Psi(e_{-1})=e_1,\qquad \Psi(e_0)=-e_0,\qquad \Psi(e_1)=e_{-1}
$$
is a Lie algebra automorphism. However, $\Psi(e_{-1})=e_1$ is not invertible in the algebra $(\mathcal W(1;1),\bullet)$, so the statement fails in this case.
\end{remark}

\begin{proposition}\label{isoq}
Let $\mathbb{F}$ be an algebraically closed field of characteristic $p>2$ and let $n\in \mathbb{N}$. Assume $(p,n)\neq (3,1)$. Given $q,q'\in \mathcal{W}(1;n)$, the algebras $\mathcal{W}_n(q)$ and $\mathcal{W}_n(q')$ are isomorphic if and only if 
there exists $\Psi\in \textrm{Aut}(\mathcal W(1;n))$ such that 
$$q'=\Psi(e_{-1})^{-2}\bullet \Psi(q). $$
where the power $\Psi(e_{-1})^{-2}$ is taken in $(\mathcal W(1;n),\bullet)$.
\end{proposition}
\begin{proof}
    Assume first that $\Psi:\mathcal{W}_n(q)\longrightarrow \mathcal{W}_n(q')$ is an isomorphism of transposed Poisson algebras. In particular, $\Psi$ is an automorphism of  $\mathcal W(1;n)$, so $\Psi(e_{-1})$ is invertible in $(\mathcal W(1;n),\bullet)$.
    Since
    $$\Psi(q) = \Psi(e_{-1} \bullet_q e_{-1}) = \Psi(e_{-1}) \bullet_{q'} \Psi(e_{-1}) = \Psi(e_{-1})^2 \bullet q',$$ 
    then we have $q'=\Psi(e_{-1})^{-2}\bullet \Psi(q).$

    Conversely, suppose there exists $\Psi\in \textrm{Aut}(\mathcal W(1;n))$ such that $q'=\Psi(e_{-1})^{-2}\bullet \Psi(q)$. Then
    $$\Psi(x \bullet_{q} y)=\Psi(x \bullet q \bullet y) = \Psi(e_{-1})^{-2}\bullet \Psi(x)\bullet \Psi(q)\bullet \Psi(y) = \Psi(x)\bullet q'\bullet \Psi(y) = \Psi(x) \bullet_{q'}\Psi(y).$$
    Hence, the map $\Psi$ is an isomorphism between $ \mathcal{W}_n(q)$ and $\mathcal{W}_n(q')$.
\end{proof}

Recall that an automorphism $\varphi$ of $\mathcal{O}(1;n)$ is called {\it admissible} if $\varphi(f^{(i)})= \varphi(f)^{(i)}$ for every non invertible $f\in \mathcal{O}(1;n)$ and $i\geq 0$. {Note that here the powers $f^{(i)}$ are given by a system of divided powers, see \cite[Definition 2.1.1]{s04}}.
Equivalently, if we assume the base field has characteristic $p>3$, $\varphi$ is admissible if and only if $\varphi(x^{(p^i)})= \varphi(x)^{(p^i)}$ for every $1 \leq i \leq n-1$, by \cite[Lemma 8]{w71}. Moreover, setting $\varphi(x)=\sum_{i=1}^{p^{n}-1} \alpha_i x^{(i)}$ with $\alpha_1\neq 0$ and $\alpha_{p^j} =0$ for $1\leq j \leq n-1$, the map $\varphi$ extends to an admissible automorphism, see \cite{w71}. %Furthermore, any admissible automorphism is of this form.

\medskip

Now, fixed $p,n >0$, for each $m$ with $0\leq m \leq p^n-1$, we introduce the set
$$E_m=
\left\{
s>0 \textrm{ with } m+s\leq p^n-1 : \ s=p^j - 1 \text{ for some } j \text{ or }\binom{m+s+1}{m}\equiv 0 \pmod p
\right\}.$$
We show that, up to isomorphism, the element $q$ can be simplified to its leading term together with the terms whose indices lie in some set $E_m$.

\begin{theorem}\label{normalform}
Let $\mathbb{F}$ be an algebraically closed field of characteristic $p>3$ and let $n\in \mathbb{N}$. Given a non-zero $q\in \mathcal{W}(1;n)$, set $m = \nu(q)+1$. Then
$$
\mathcal{W}_n(q)\cong \mathcal{W}_n\!\left(e_{\nu(q)}+\sum_{i\in E_m} a_i e_{\nu(q)+i}\right),
$$
for certain scalars $a_i\in \mathbb{F}$ with $i\in E_m$.
\end{theorem}
\begin{proof}
    We use the characterization given in Proposition~\ref{isoq}. 
    In order to prove that $\mathcal{W}_n(q)$ is
isomorphic to $\mathcal{W}_n(q')$ with $q'$ as in the statement, we have to prove that there is $\Psi\in \textrm{Aut}(\mathcal W(1;n))$ such that $q'=\Psi(e_{-1})^{-2}\bullet \Psi(q)$. By \cite[Theorem 12.8]{ree}, there is an admissible automorphism $\varphi\in \textrm{Aut}(\mathcal{O}(1;n))$ such that  $\Psi(D) = \varphi D \varphi^{-1}$ for every $D\in \mathcal W(1;n)$.   Therefore, the algebras $\mathcal{W}_n(q)$ and $\mathcal{W}_n(q')$ are isomorphic if there exists an automorphism $\varphi_y\in \textrm{Aut}(\mathcal{O}(1;n))$ such that $\tilde{q'}=\partial(\varphi_y(x))\varphi_y(\tilde{q})$, where $\varphi_y$ is determined by
    $$\varphi_y(x)=y=\sum_{i=1}^{p^{n}-1} \alpha_i x^{(i)} \qquad \textrm{ with } \alpha_1\neq 0 \textrm{ and } \alpha_{p^j} =0 \textrm{ for } 1\leq j \leq n-1.$$ 
    
    Denote $m = \nu(q) + 1$ and $N=p^n$. Then we can write
    $$\tilde q=q_mx^{(m)}+q_{m+1}x^{(m+1)}+\cdots + q_{N-1}x^{(N-1)}
\qquad\text{with }q_m\neq 0.$$
   
    First, we prove that the leading coefficient can be normalized. We can perform the following operation that give us an element corresponding to an isomorphic algebra. Indeed, set $y = \lambda x^{(1)}$ with $\lambda\in\mathbb{F}$ such that $\lambda^{m+1} q_m = 1$. Then it follows
    \begin{equation*}  \partial(\varphi_y(x))\varphi_y(\tilde{q}) = \lambda (\lambda^{m} q_m x^{(m)} + \textrm{terms of degree} > m) = x^{(m)} + \textrm{terms of degree} > m. 
    \end{equation*}
    Hence, we obtain an element with normalized leading term corresponding to an isomorphic algebra. So from here assume without loss of generality that $\tilde{q}$ is normalized, that is $q_m=1$.

    Now, fix $s$ such that $1\leq s\leq N-1-m$ and suppose $s\notin E_m$. We can define $y = x^{(1)} + \lambda x^{(s+1)}$ for $\lambda\in \mathbb{F}$, because $s+1\neq p^j$. Then, if $m>0$, we have
    \begin{equation*}
        \begin{split}        \partial(\varphi_y(x))\varphi_y(\tilde{q}) &= (1 + \lambda x^{{(s)}})(\tilde{q} + \lambda x^{(m-1)}x^{(s+1)} + \textrm{terms of degree} > m+s) \\
        &=\tilde{q} + \binom{m+s}{s+1} \lambda x^{(m+s)} +  \binom{m+s}{s} \lambda x^{(m+s)} + \textrm{terms of degree} > m+s.
        \end{split}
    \end{equation*}
    Since $s\not\in E_m$, we can choose $\lambda$ such that $$q_{m+s} + \binom{m+s+1}{m} \lambda = 0.$$
    This choice eliminates the term corresponding to $x^{(m+s)}$ in the previous formula without changing
the terms of smaller degrees. 
Also, if $m=0$, then 
    \begin{equation*}       \partial(\varphi_y(x))\varphi_y(\tilde{q}) = \tilde{q} + \lambda x^{(s)} + \textrm{terms of degree} > s,
    \end{equation*}
and the elimination follows by setting $\lambda$ such that $\lambda = -q_s$.

Finally, we repeat this elimination process for every $s\notin E_m$, starting
from lower degrees and proceeding to higher degrees. At each step, only the
coefficient of $x^{(m+s)}$ and terms of strictly higher degree are modified, so
terms already eliminated in lower degrees do not reappear.
After finitely many steps, we obtain an element $q'$ in the form of the statement.
\end{proof}

\begin{corollary}
    Let $\mathbb{F}$ be an algebraically closed field of characteristic $p>3$ and let $n\in \mathbb{N}$. Given an invertible $q\in \mathcal{W}(1;n)$, then
$$
\mathcal{W}_n(q)\cong \mathcal{W}_n\!\left(e_{-1}+\sum_{j=1}^{n} a_{j} e_{p^j-2}\right),
$$
for certain scalars $a_1, \ldots, a_n\in \mathbb{F}$.
\end{corollary}
\begin{proof}
    The result follows by applying Theorem~\ref{normalform} for $m=\nu(q) + 1 = 0$ and determining $E_0$.
\end{proof}

\begin{corollary}
Let $\mathbb{F}$ be an algebraically closed field of characteristic $p>3$. Given a non-zero $q\in \mathcal{W}(1;1)$, then
\begin{enumerate}
\item if $\nu(q)<p-2$, there exists $a\in\mathbb{F}$ such that
$$\mathcal W_1(q)\cong \mathcal W_1(e_{\nu(q)}+a e_{p-2});$$
\item if $\nu(q)=p-2$, then
$\mathcal W_1(q)\cong \mathcal W_1(e_{p-2}).$
\end{enumerate}
\end{corollary}
\begin{proof}
    It is easy to show that in this setting     $E_m = \left\{p-m-1\right\}$ for every $m$ such that $0\leq m\leq p-2$ and $E_{p-1}$ is the empty set. Therefore, we obtain the stated classification applying Theorem~\ref{normalform}.
\end{proof}

\section{Irreducible representations of the simple transposed Poisson algebras}\label{sec5}

In this section, we study the irreducible representations of the simple transposed Poisson algebras of the family $\mathcal{W}_n(q)$. To the best of our knowledge, this is the first study of irreducible representations for a concrete family of transposed Poisson algebras.

\subsection{Representations of transposed Poisson algebras}

Let us recall the definition of a representation of transposed Poisson algebras introduced in \cite{bitpa}.

 \begin{definition}
    Let $(\mathcal{P}, \circ, [\cdot, \cdot])$ be a transposed Poisson algebra. A representation of the transposed Poisson algebra $\mathcal{P}$ is a triple $(\alpha, \beta, V)$ such that 
    \begin{enumerate}
        \item $(\alpha, V)$ is a representation of the associative commutative algebra $(\mathcal{P}, \circ)$, 
        \item $(\beta, V)$ is a representation of the Lie algebra $(\mathcal{P}, [\cdot, \cdot])$,
        \item and the following compatibility conditions hold for all $x, y \in \mathcal{P}$
        \begin{equation}\label{repcomp1}
      2\alpha(x)\beta(y) = \beta(x\circ y) + \beta(y)\alpha(x),
        \end{equation}
        \begin{equation}\label{repcomp2}
      2\alpha([x, y]) = \beta(x)\alpha(y) - \beta(y)\alpha(x).
        \end{equation}
    \end{enumerate}
\end{definition}

Due to the rich class of identities satisfied by transposed Poisson algebras (see \cite{tpa} and the first section), we have numerous relations satisfied by its representations.  In the following result, we enumerate some useful ones. The proof is straightforward.

\begin{lemma}\label{lm52}
    Let $(\mathcal{P}, \circ, [\cdot, \cdot])$ be a transposed Poisson algebra and let $(\alpha, \beta, V)$ be a representation of $\mathcal{P}$. Then the following relations hold for all $x, y, z\in \mathcal{P}$.
    \begin{equation}\label{rel1}
            \alpha(x)\beta(y)-\alpha(y)\beta(x) + \alpha([x, y])= 0,
    \end{equation}
%    \begin{equation}\label{rel2}
%         \beta(x)\alpha([y, z]) + \beta(y)\alpha([z, x]) + \beta(z)\alpha([x, y])=0
%    \end{equation}    
    \begin{equation}\label{rel3}
        \beta([x, y])\alpha(z) + \beta([y, z])\alpha(x) + \beta([z, x])\alpha(y) = 0,
    \end{equation}
    \begin{equation}\label{rel3b}
        \beta([x, y])\alpha(z) = \beta(z\circ x)\beta(y) - \beta(z\circ y)\beta(x).
   \end{equation}
%    \begin{equation}\label{rel4}
%        \alpha([y, z])\beta(h) + \alpha([h, y])\beta(z) - \alpha([h, z]) \beta(y)=0
%   \end{equation}
%    \begin{equation}\label{propeq4}
%        [h, x][y, z] + [h, y][z, x] + [h, z][x, y]= 0.
%    \end{equation}
%    \begin{equation}\label{propeq5}
%        [xu, vy] + [xv, uy] = 2 uv[x, y].
%    \end{equation}
%    \begin{equation}\label{propeq6}
%        x[u, yv] + v[xy, u] + yu[v, x] = 0.
%    \end{equation}
\end{lemma}

As usual, a {\it subrepresentation} of a representation $(\alpha, \beta, V)$ of a transposed Poisson algebra is a triple $(\alpha'= \alpha|_{V'}, \beta' = \beta|_{V'}, V')$, with $V'\subseteq V$, such that $\alpha(\mathcal{P})(V') \subseteq V'$ and $\beta(\mathcal{P})(V')\subseteq V'$. A representation of the transposed Poisson algebra is {\it irreducible} if it has no proper subrepresentations. Clearly, if $(\alpha, V)$ is an irreducible representation of an associative commutative algebra or if $(\beta, V)$ is an irreducible representation of Lie algebras, then $(\alpha, \beta, V)$ is an irreducible representation of transposed Poisson algebras. Conversely, the next result holds for transposed Poisson algebras with perfect Lie bracket. Notice that simple transposed Poisson algebras have perfect Lie bracket.

\begin{proposition}\label{irrLie}
    Let $(\mathcal{P}, \circ, [\cdot, \cdot])$ be a transposed Poisson algebra with perfect Lie bracket and let $(\alpha, \beta, V)$ be an irreducible representation of $\mathcal{P}$ with $\beta\neq 0$. Then $(\beta, V)$ is an irreducible representation of $(\mathcal{P}, [\cdot, \cdot])$.
\end{proposition}
\begin{proof}
    Let $(\alpha, \beta, V)$ be an irreducible representation of $\mathcal{P}$. Then by the relation (\ref{repcomp1}), we have
    $$\alpha(\mathcal{P})\beta(\mathcal{P})(V) \subseteq \beta(\mathcal{P}\circ \mathcal{P})(V) + \beta(\mathcal{P})\alpha(\mathcal{P})(V) \subseteq \beta(\mathcal{P})(V).$$
    This implies that $(\alpha|_{V'}, \beta|_{V'}, V')$,
    where $V'=\beta(\mathcal{P})(V)$, is a subrepresentation of $(\alpha, \beta, V)$. Since $\beta\neq0$, it follows that $\beta(\mathcal{P})(V) = V$. Suppose there is a subspace $W\subseteq V$ with $V\neq W$ such that $\beta(\mathcal{P})(W) \subseteq W$. Denote by $W'$ the maximal subspace of $V$ such that $\beta(\mathcal{P})(W')= W$. Note that $W'\neq V$. Since $[\mathcal{P}, \mathcal{P}] = \mathcal{P}$, we have  by the equation (\ref{rel3b}) that $$\beta(\mathcal{P})\alpha(\mathcal{P})(W') \subseteq \beta([\mathcal{P}, \mathcal{P}])\alpha(\mathcal{P})(W') \subseteq \beta(\mathcal{P}\circ \mathcal{P})\beta(\mathcal{P}) (W') - \beta(\mathcal{P}\circ \mathcal{P})\beta(\mathcal{P}) (W') \subseteq W.$$
    By the relation (\ref{repcomp1}), it follows that 
    $$\alpha(\mathcal{P})(W) = \alpha(\mathcal{P})\beta(\mathcal{P})(W') \subseteq \beta(\mathcal{P}\circ \mathcal{P})(W') + \beta(\mathcal{P})\alpha(\mathcal{P})(W') \subseteq \beta(\mathcal{P})(W') + W = W.$$
    Consequently, $(\alpha|_{W}, \beta|_{W}, W)$ is a subrepresentation which is a contradiction. Hence, $(\beta, V)$ must be irreducible.
\end{proof}

In fact, the representation $\alpha$ of $(\mathcal{P},\circ)$ is prescribed in this case.

\begin{proposition}\label{uab}
Let $(\mathcal{P}, \circ, [\cdot, \cdot])$ be a transposed Poisson algebra with perfect Lie bracket and let $(\beta, V)$ be a representation of the Lie algebra $(\mathcal{P}, [\cdot, \cdot])$. Then there exists at most one representation $(\alpha, V)$ of the associative commutative algebra $(\mathcal{P}, \circ)$ such that $(\alpha, \beta, V)$ is a representation of the transposed Poisson algebra $\mathcal{P}$.
\end{proposition}
\begin{proof}
Assume that $\alpha$ and $\alpha'$ are two representations of $(\mathcal P,\circ)$ such that both
$(\alpha,\beta,V)$ and $(\alpha',\beta,V)$ are representations of the transposed Poisson algebra
$\mathcal P$. We show that $\alpha=\alpha'$.
Since the Lie bracket is perfect, it is enough to show that they coincide on every element in 
$[[\mathcal P, \mathcal P],[\mathcal P, \mathcal P]]$.

Using the compatibility relation \eqref{repcomp2}
and equation \eqref{rel3b} for every $w, x, y, z \in \mathcal{P}$ we have 
\begin{equation}
\begin{split}\label{eqrep}
2\alpha([[w,x],[y,z]])
= \,\, & \beta([w,x])\alpha([y,z])-\beta([y,z])\alpha([w,x])\\
= \,\, & \beta([y,z]\circ w)\beta(x)-\beta([y,z]\circ x)\beta(w) \\
 & - \beta([w,x]\circ y)\beta(z)+\beta([w,x]\circ z)\beta(y).
\end{split}
\end{equation}

The right-hand side depends only on $\beta$ and the two multiplications of $\mathcal P$.
Obviously, the same formula holds for $\alpha'$, so we deduce
$$
\alpha([[w,x],[y,z]])=\alpha'([[w,x],[y,z]])
$$

Since such elements span $\mathcal P$, we conclude that $\alpha=\alpha'$.
\end{proof}

\begin{proposition}\label{alfa1}
Let $(\mathcal{P}, \circ, [\cdot, \cdot])$ be a transposed Poisson algebra  with perfect Lie bracket. Let $(\alpha, \beta, V)$ be an irreducible representation of $\mathcal{P}$ with $\beta\neq 0$. If $\mathcal{P}$ is unital, 
then $\alpha(1)=\textrm{Id}_V$. 
\end{proposition}
\begin{proof}
    By equation \eqref{rel3b}, we have that
    $$\beta([x, y])\alpha(1) = \beta(x)\beta(y) - \beta(y)\beta(x) = \beta([x, y]).$$
    The Lie bracket is perfect, so we deduce $\beta(x)\alpha(1) = \beta(x)$. Also, by equation \eqref{repcomp1}, we obtain
    $$2\alpha(1)\beta(x) = \beta(x) + \beta(x)\alpha(1) = 2\beta(x).$$
    Hence, we have $\alpha(1)\beta(x)=\beta(x)$ for every $x\in \mathcal{P}$. Since $(\beta, V)$ is irreducible by Proposition~\ref{irrLie}, it follows that $\beta(\mathcal{P})(V) = V$ and $\alpha(1) = \textrm{Id}_V$.    
\end{proof}

Notice that if a simple transposed Poisson algebra $(\mathcal{P}, \circ, [\cdot, \cdot])$ is trivial, that is, if one of the two multiplications is zero, then the representation theory simplifies considerably. If the associative multiplication is zero, then by \eqref{eqrep} we obtain $\alpha=0$, so the irreducible representations are precisely the trivial one-dimensional representation and the irreducible representations of the Lie algebra $(\mathcal{P}, [\cdot, \cdot])$. If the Lie bracket is zero, then $\beta=0$, and the only irreducible representation is the one-dimensional trivial one. Therefore, in the current setting, it remains to study the irreducible representations of the non-trivial simple transposed Poisson algebras, namely the algebras $\mathcal{W}_n(q)$.

\subsection{Irreducible representations of $\mathcal{W}_n(q)$}
In this subsection, we determine the irreducible finite-dimensional representations of $\mathcal{W}_n(q)$ when it is unital. 
Recall that, by Lemma~\ref{invuni}, $\mathcal{W}_n(q)$ is unital if and only if $q$ is invertible. Let us assume $V$ is a finite-dimensional vector space. 
Also, denote $\delta = Q_{e_{-1}}$ for the distinguished derivation of $\mathcal{W}_n(e_{-1})$ and recall we write $N = p^n$. 
We have the following consequence of equation \eqref{eqrep}.

\begin{corollary}\label{beta0}
Let $(\alpha,\beta,V)$ be an irreducible representation of
$\mathcal{W}_n(q)$. If $\beta=0$, then $\alpha=0$ and $\textrm{dim}(V)=1$.
\end{corollary}

Hence, it remains to study the case $\beta\neq 0$. First, we study the irreducible representations on the model algebra $\mathcal{W}_n(e_{-1})$. We have the following key relations.

\begin{lemma}\label{relsbeta1}
    Let $(\alpha,\beta,V)$ be an irreducible representation of
$\mathcal{W}_n(e_{-1})$. Denote $T=\beta(e_{-1})$. Then for every $x\in \mathcal{W}_n(e_{-1})$, we have
$$\beta(x) = \alpha(x)T - \alpha(\delta(x)) \qquad \textrm{ and } \qquad \beta(x) = T \alpha(x) - 2 \alpha(\delta(x)).$$
Moreover, we have  $\alpha(x)T=T\alpha(x)-\alpha(\delta(x))$. 
\end{lemma}

\begin{proof}
    The relations follow directly from Proposition~\ref{alfa1} and the relations in Lemma~\ref{lm52}.
\end{proof}

Let $v_0\in V$ be a common eigenvector of the commuting linear maps $\alpha(\mathcal{W}_n(e_{-1}))$, that is, $\alpha(x)(v_0)= \lambda(x) v_0$ for some $\lambda\in \mathcal{W}_n(e_{-1})^*$. Denote $v_k = T^k v_0$ for $k\geq 0$ and consider the vector subspace $V_0$ spanned by the vectors $v_k$.

\begin{lemma}
    Let $(\alpha,\beta,V)$ be an irreducible representation of
$\mathcal{W}_n(e_{-1})$. Then  $V=V_0$.
\end{lemma}
\begin{proof}
    First, we show that $V_0$ is $\alpha$-stable. Indeed, we prove that $\alpha(x)(T^kv_0)\in V_0$ for every $x\in \mathcal{W}_n(e_{-1})$ by induction on $k\geq 0$. For $k=0$, the result is clear by the choice of $v_0$. Now, assume that $\alpha(\mathcal{W}_n(e_{-1}))(T^k v_0)\subseteq V_0$. Then, using the relation $\alpha(x)T=T\alpha(x)-\alpha(\delta(x))$, we obtain
    $$\alpha(x)(T^{k+1}v_0)
=\alpha(x)T(T^k v_0)
=T\alpha(x)(T^k v_0)-\alpha(\delta(x))(T^k v_0).$$
By the induction hypothesis, both
$\alpha(x)(T^k v_0)\in V_0$ and $\alpha(\delta(x))(T^k v_0)\in V_0$. Since $TV_0 \subseteq V_0$, it follows $\alpha(x)(T^{k+1}v_0) \in V_0$. Hence, the space $V_0$ is $\alpha$-stable.

    Moreover, observe that $V_0$ is $\beta$-stable. Since 
    $\beta(x)(T^kv_0) = \alpha(x)(T^{k+1}v_0) - \alpha (\delta(x))(T^kv_0)$, by Lemma~\ref{relsbeta1}, and the fact that $V_0$ is $\alpha$-stable, it follows that $\beta(x)(T^kv_0) \in V_0$.

    By the irreducibility of $(\alpha,\beta,V)$ and since $V_0\neq 0$, we conclude that $V=V_0$.
\end{proof}

Let $(\mathcal{L}, [\cdot, \cdot])$ be a Lie algebra over a field $\mathbb{F}$. Recall that a {\it self-centralizing} element $x\in \mathcal{L}$ is one such $\left\{a\in \mathcal{L}: [x, a] = 0\right\} = \mathbb{F} x$ and an {\it ad-nilpotent} element $x\in \mathcal{L}$ is one such that $Q_x$ is nilpotent, see \cite{BIO}. We prove the following auxiliary lemma.

\begin{lemma}\label{selfcent}
    The element $e_{-1}\in \mathcal{W}_n(q)$ is self-centralizing and ad-nilpotent.
\end{lemma}
\begin{proof}
Since the Lie bracket of $\mathcal W_n(q)$ coincides with that of $\mathcal W(1;n)$, it is enough to prove the statement in $\mathcal W(1;n)$. By \eqref{usual}, we have $[e_{-1},e_j]=e_{j-1}$ for $-1\leq j\leq p^n-2$. 
If $x=\sum_{i=-1}^{N-2}\alpha_i e_i$ and $[e_{-1},x]=0$, then
$$
0=[e_{-1},x]=\sum_{i=0}^{p^n-2}\alpha_i e_{i-1},
$$
so $\alpha_i=0$ for all $i\ge 0$. Hence, $e_{-1}$ is self-centralizing. Since $Q_{e_{-1}}(e_j)=e_{j-1}$, the operator $Q_{e_{-1}}$ is nilpotent. Therefore, $e_{-1}$ is ad-nilpotent.
\end{proof}

By Lemma~\ref{selfcent}, we know $e_{-1}$ is a self-centralizing ad-nilpotent element of $\mathcal{W}_n(e_{-1})$.
Hence, $\delta$ has a unique Jordan block and we may choose a basis
$$
f_{-1}=e_{-1},\ f_0,\dots,f_{N-2}
$$
of $\mathcal{W}_n(e_{-1})$ such that $\lambda(f_i) = 0$ and
$\delta(f_i)=f_{i-1} $ for $i\geq 0$. 

\begin{lemma}\label{actionalpha}
For every $0\leq i\leq N-2$ and every $k\geq 0$, we have
$$
\alpha(f_i)(v_k)=(-1)^{i+1}\binom{k}{i+1}v_{k-i-1},
$$
where we use the convention $v_s=0$ if $s<0$.
\end{lemma}

\begin{proof}
Proceed by induction on $k\geq 0$.
If $k=0$, then $\alpha(f_i)(v_0)=\lambda(f_i)v_0=0
$ for $i\geq 0$, while
$$
(-1)^{i+1}\binom{0}{i+1}v_{-i-1}=0.
$$
So the equation holds.
Assume now that the formula holds up to $k\geq 0$. Then we obtain
\begin{equation*}
    \begin{split}
        \alpha(f_i)(v_{k+1})
&=\alpha(f_i)(Tv_k)
=T\alpha(f_i)(v_k)-\alpha(f_{i-1})(v_k)\\
&= (-1)^{i+1}\binom{k}{i+1}v_{k-i} - (-1)^i\binom{k}{i}v_{k-i} =
(-1)^{i+1}\binom{k+1}{i+1}v_{k-i},
    \end{split}
\end{equation*}
which completes the induction.
\end{proof}

We now identify the associative action $\alpha$.
Recall that a representation $(\alpha,V)$ of an associative algebra
$(\mathcal A,\circ)$ is called the {\it regular representation} if $V=\mathcal A$ and
$
\alpha(x)=P_x
$
for every $x\in \mathcal A$. 

\begin{proposition}
Let $(\alpha,\beta,V)$ be an irreducible representation of
$\mathcal{W}_n(e_{-1})$. Then $(\alpha,V)$ is isomorphic to the regular representation.
\end{proposition}

\begin{proof}
Denote $m=\textrm{dim}(V)$. Since $V=V_0$, we may choose the minimal $m$ such that
$V$ is spanned by $\left\{v_0,\dots,v_{m-1}\right\}$. Then $\left\{v_0,\dots,v_{m-1}\right\}$ is a basis of $V$.
We first prove that $N$ divides $m$. By Lemma~\ref{actionalpha}, for every $0\leq i\leq N-2$ we have
$$
\alpha(f_i)(v_m)=(-1)^{i+1}\binom{m}{i+1}v_{m-i-1}.
$$
Since $v_m\in \textrm{span}(v_0,\dots,v_{m-1})$, the left hand side belongs to
$\textrm{span}(v_0,\dots,v_{m-i-2})$. Therefore
$$
\binom{m}{i+1}\equiv 0 \pmod p
$$
for every $0\leq i\leq N-2$. 
Write $m = \sum_{k=0}^{r} m_k p^k$ in base $p$. Then for $0\leq j\leq n-1$ it follows
$$
m_j = \binom{m_j}{1} \equiv \binom{m}{p^j}\equiv 0 \pmod p,
$$
by Lucas' theorem.
We deduce $N\mid m$.
Moreover, we claim $U=\textrm{span}(v_{m-N},v_{m-N+1},\dots,v_{m-1})$ is $\alpha$-stable and $\beta$-stable. 
Indeed, for $m-N\leq k\leq m-1$ and $0\leq i\leq N-2$, Lemma~\ref{actionalpha} gives
$$
\alpha(f_i)(v_k)=(-1)^{i+1}\binom{k}{i+1}v_{k-i-1}.
$$
If $k-i-1\ge m-N$, then $\alpha(f_i)(v_k)\in U$. Otherwise, $k-i-1<m-N$, so
$i+1>k-(m-N).$

Since $N\mid m$ and $N\mid (m-N)$, Lucas' theorem gives
$$
\binom{k}{i+1}
=
\binom{m-N+(k-(m-N))}{i+1}
\equiv
\binom{k-(m-N)}{i+1}
\equiv 0 \pmod p.
$$
Hence, we have $\alpha(f_i)(v_k)\in U$. Thus, $U$ is $\alpha$-stable.
Now, Lemma~\ref{relsbeta1}  and $\alpha(f_i)(v_m) = 0$ yields
$$
\beta(f_i)(v_k)=\alpha(f_i)(Tv_k)-\alpha(\delta(f_i))(v_k)
=\alpha(f_i)(v_{k+1})-\alpha(f_{i-1})(v_k)\in U,
$$
so $U$ is also $\beta$-stable. Hence, the irreducibility of $(\alpha, \beta, V)$ implies $U=V$ and $m=N$.
\medskip

Now, consider the linear map $\varphi:\mathcal{W}_n(e_{-1})\to V$ given by
$
\varphi(x)=\alpha(x)(v_{N-1}).
$
Since $(\alpha,V)$ is a representation of the associative algebra $(\mathcal{W}_n(e_{-1}),\bullet)$, $\varphi$ is a homomorphism between the regular representation and $(\alpha,V)$.
For $0\leq k\leq N-1$, we have
$$
\varphi(f_{N-k-2})
=
\alpha(f_{N-k-2})(v_{N-1})
=
(-1)^{N-k-1}\binom{N-1}{N-k-1}v_k.
$$
Then there is $t\in \mathbb{F}$ such that $\varphi(tf_{N-k-2}) = v_k$, because we have
$$
\binom{N-1}{N-k-1} = \binom{p^n-1}{p^n-k-1}  \not\equiv 0 \pmod p.
$$
Hence $\varphi$ is surjective. Since both spaces have dimension $N$,  the map $\varphi$ is an isomorphism.
Therefore, $(\alpha,V)$ is isomorphic to the regular representation of $(\mathcal{W}_n(e_{-1}),\bullet)$.
\end{proof}

We assume $(\alpha,V)$ is the regular representation and determine the Lie action.

\begin{proposition}\label{betaclass}
Suppose that $V=\mathcal{W}_n(e_{-1})$ and $\alpha(x)=P_x$ for every $x\in \mathcal{W}_n(e_{-1})$. Then there exists a unique element $a\in \mathcal{W}_n(e_{-1})$ such that
$$
\beta(x)=Q_x+P_{x\bullet a}
$$
for every $x\in \mathcal{W}_n(e_{-1})$.
\end{proposition}

\begin{proof}
We have
$\beta(x)=P_xT-P_{\delta(x)}$ for every $x\in \mathcal{W}_n(e_{-1})$, by Lemma~\ref{relsbeta1}.
Also, since $\mathcal{W}_n(e_{-1})$ is unital, we have 
$Q_x=P_x\delta -P_{\delta(x)}.$
Therefore, $\beta(x)-Q_x=P_x(T-\delta)$
for every $x\in \mathcal{W}_n(e_{-1})$.
Denote $a = \beta(e_{-1})(e_{-1})$. Since we have $P_xT = T P_x - P_{\delta(x)}$ by Lemma~\ref{relsbeta1}, we obtain
$$(T-\delta)(x)=(T-\delta)P_x(e_{-1})=P_x(a) = P_a(x).$$
Hence, it follows $\beta(x)= Q_x + P_xP_a = Q_x+P_{x\bullet a}$. The uniqueness follows by the choice of $a$. 
\end{proof}

For every $a\in \mathcal{W}_n(e_{-1})$, consider the tuple
$M(a)=\bigl(P,\beta_a,\mathcal{W}_n(e_{-1})\bigr)$
where $\beta_a(x)=Q_x+P_{x\bullet a}.$

\begin{lemma}\label{Mairr}
For every $a\in \mathcal{W}_n(e_{-1})$, the tuple $M(a)$ is an irreducible representation of $\mathcal{W}_n(e_{-1})$.
\end{lemma}
\begin{proof}
A straightforward verification of the definition shows that $M(a)$ is a representation of $\mathcal{W}_n(e_{-1})$.  Now, let $W\subseteq \mathcal{W}_n(e_{-1})$ be a subrepresentation of $M(a)$. Since $W$ is stable under $P_x$ and $\beta_a(x)$, it is also stable under
$Q_x=\beta_a(x)-P_{x\bullet a}$. 
Hence $W$ is an ideal of the transposed Poisson algebra $\mathcal{W}_n(e_{-1})$, which is a contradiction unless $W=0$ or $W = \mathcal{W}_n(e_{-1})$.
\end{proof}

At this point, we can deduce the classification of irreducible modules on $\mathcal{W}_n(q)$ for $q=e_{-1}$ from the previous results. But, moreover, the next correspondence allows us to obtain the classification for every invertible $q$ in $(\mathcal W(1;n),\bullet)$.

\begin{lemma}\label{twistrep}
Let $q\in \mathcal W(1;n)$ be invertible in $(\mathcal W(1;n),\bullet)$, and write
$u=q^{-1}$. Let $(\alpha,\beta,V)$ be a representation of the transposed Poisson algebra
$\mathcal W_n(q)$.
Define
$$
\widetilde{\alpha}(x)=\alpha(x\bullet u),\qquad x\in \mathcal W(1;n).
$$
Then $(\widetilde{\alpha},\beta,V)$ is a representation of the transposed Poisson algebra
$\mathcal W_n(e_{-1})$.

Conversely, if $(\widetilde{\alpha},\beta,V)$ is a representation of $\mathcal W_n(e_{-1})$, then
$$
\alpha_q(x)=\widetilde{\alpha}(x\bullet q),\qquad x\in \mathcal W(1;n),
$$
defines a representation $(\alpha_q,\beta,V)$ of $\mathcal W_n(q)$.

Moreover, irreducibility is preserved by this correspondence.
\end{lemma}

\begin{proof}
Suppose $(\alpha,\beta,V)$ is a representation of $\mathcal W_n(q)$ and set
$\widetilde{\alpha}(x)=\alpha(x\bullet u)$.
First, notice that 
$$
\widetilde{\alpha}(x)\widetilde{\alpha}(y)
=\alpha(x\bullet u)\alpha(y\bullet u) = \alpha((x\bullet u) \bullet_q (y\bullet u))
=\alpha(x\bullet y\bullet u)
=\widetilde{\alpha}(x\bullet y).
$$
Hence $(\widetilde{\alpha},V)$ is a representation of the associative algebra
$(\mathcal W(1;n),\bullet)$.

Next, using the compatibility identities for $(\alpha,\beta,V)$, we obtain
$$
2\widetilde{\alpha}(x)\beta(y)
=2\alpha(x\bullet u)\beta(y)
=\beta((x\bullet u)\bullet_q y)+\beta(y)\alpha(x\bullet u)
=\beta(x\bullet y)+\beta(y)\widetilde{\alpha}(x),
$$
$$
2\widetilde{\alpha}([x,y])
=2\alpha([x,y]\bullet u)
=\beta(x)\alpha(y\bullet u)-\beta(y)\alpha(x\bullet u)
=\beta(x)\widetilde{\alpha}(y)-\beta(y)\widetilde{\alpha}(x).
$$
Therefore, $(\widetilde{\alpha},\beta,V)$ is a representation of $\mathcal W_n(e_{-1})$.

The converse is proved in the same way. Finally, since the map
$x\mapsto x\bullet u$ is invertible, we have
$\widetilde{\alpha}(\mathcal W(1;n))=\alpha(\mathcal W(1;n)),$
so the invariant subspaces for $(\alpha,\beta,V)$ and $(\widetilde{\alpha},\beta,V)$ coincide.
Hence, the irreducibility is preserved.
\end{proof}

For every invertible $q\in \mathcal W(1;n)$ and every $a\in \mathcal W(1;n)$, denote by
$
M_q(a)=\bigl(\alpha_q,\beta_a,\mathcal W(1;n)\bigr)
$
the representation given by
$$
\alpha_q(x)=P_{x\bullet q}
\qquad\text{and}\qquad
\beta_a(x)=Q_x+P_{x\bullet a}.
$$
By Lemmas~\ref{twistrep} and \ref{Mairr}, each $M_q(a)$ is an irreducible representation of $\mathcal W_n(q)$.

\medskip

The arguments above prove the following classification result.

\begin{theorem}\label{classificationunitalreps2}
Let $\mathbb F$ be an algebraically closed field of characteristic $p>2$, let $n\in \mathbb N$, and let
$q\in \mathcal W(1;n)$ be invertible. Then every finite-dimensional irreducible representation of
$\mathcal{W}_n(q)$ is
\begin{enumerate}
    \item the one-dimensional trivial representation, or
    \item a representation of the form {$M_q(a)$ for some $a\in \mathcal W(1;n)$}.
\end{enumerate}
Moreover, every representation {$M_q(a)$} has dimension $p^n$.
\end{theorem}
\begin{proof}
Let $(\alpha,\beta,V)$ be a finite-dimensional irreducible representation of $\mathcal W_n(q)$.
If $\beta=0$, then the representation is one-dimensional trivial by Corollary~\ref{beta0}.
Assume now that $\beta\neq 0$. By Lemma~\ref{twistrep}, the triple $(\widetilde{\alpha},\beta,V),$ with $\widetilde{\alpha}(x)=\alpha(x\bullet q^{-1}),$
is an irreducible representation of $\mathcal W_n(e_{-1})$. By the previous propositions, this representation is isomorphic to $M(a)$ for $a\in \mathcal W(1;n)$. Applying Lemma~\ref{twistrep} again, we obtain that $(\alpha,\beta,V)$ is isomorphic to $M_q(a)$.
\end{proof}

This completes the classification in the unital case. The non-unital case remains open. Although the associative multiplication is again obtained by a mutation of $(\mathcal W(1;n),\bullet)$, the arguments above rely essentially on {the invertibility of $q$}. Therefore, the non-unital case requires a separate analysis and should be considered in future work.

\section{Further consequences and remarks}\label{sec6}
In this section, we record some consequences of the classification obtained above and discuss further properties of the family $\mathcal{W}_n(q)$, including its relation with weak-Leibniz algebras, Jordan superalgebras and quasi-Poisson algebras.

\subsection{Simple weak-Leibniz algebras in prime characteristic}

The notion of a weak-Leibniz algebra was introduced in \cite{askar} as a generalization of Leibniz algebras. Namely, a {\it weak-Leibniz} algebra $(\mathcal{L},\cdot)$ is an algebra satisfying the identities for all $x, y, z\in\mathcal{L}$
$$(x\cdot y)\cdot z - (y\cdot x)\cdot z = 2x\cdot (y\cdot z)-2y\cdot (x\cdot z) \textrm{ \quad and \quad} x\cdot (y\cdot z)-x\cdot (z\cdot y) = 2(x\cdot y)\cdot z-2(x\cdot z)\cdot y.$$ 

Weak-Leibniz algebras are a generalization of Leibniz algebras with the following property. Given a transposed Poisson algebra $(\mathcal{P}, \circ, [\cdot, \cdot])$, we can consider a new multiplication $\cdot$ in $\mathcal{P}$, defined by $x\cdot y = x\circ y + [x, y]$ for any $x, y \in \mathcal{P}$. The algebra $(\mathcal{P}, \cdot)$ is called the depolarization of $\mathcal{P}$. Conversely, given a weak-Leibniz algebra $(\mathcal{L}, \cdot)$, we can consider the multiplications $\circ$ and $[\cdot, \cdot ]$ in $\mathcal{L}$, given by $x \circ y = \frac{1}{2}(x\cdot y + y\cdot x)$ and $[x, y]=\frac{1}{2}(x\cdot y - y\cdot x)$, for any $x, y \in \mathcal{L}$. This new algebra $(\mathcal{L}, \circ, [\cdot, \cdot])$ is called the polarization of $\mathcal{L}$.
Dzhumadil'daev  showed that any depolarized transposed Poisson algebra is a weak-Leibniz algebra and that any polarized weak-Leibniz algebra is a transposed Poisson algebra for $p>3$. Note that the simplicity is preserved by polarization and depolarization. This implies that any simple finite-dimensional complex weak-Leibniz algebra is an associative commutative algebra or a Lie algebra \cite[Corollary 17]{simple}. In the case of prime characteristic, the description follows as a consequence of Theorem \ref{clasip}.

\begin{corollary}
    Let $\mathbb{F}$ be algebraically closed and let $\textrm{char}(\mathbb{F}) = p>3$. Then any simple finite-dimensional weak-Leibniz algebra is an associative commutative algebra or a Lie algebra or an algebra isomorphic to the depolarization of one of the algebras $\mathcal{W}_n(q)$ of Theorem~\ref{clasip}.
\end{corollary}

\subsection{Speciality of the Kantor double of $\mathcal{W}_n(q)$} 
Let $(\mathcal P,\circ,[\cdot,\cdot])$ be a transposed Poisson algebra. The Kantor double (see \cite{kantor92, m92}) of the algebra $\mathcal P$ is the $\mathbb Z_2$-graded vector space
$\mathfrak J(\mathcal P)=\mathcal P\oplus \mathcal P^s$
whose even part is $\mathcal P$ and whose odd part is a copy $\mathcal P^s$ of $\mathcal P$, endowed with the multiplication given, for all $x,y\in \mathcal P$, by
$$
x*y=x\circ y,\qquad x*y^s=(x\circ y)^s,\qquad x^s*y=(x\circ y)^s,\qquad x^s*y^s=[x,y].
$$    
It was shown in \cite{simple} that the Kantor double of a transposed Poisson algebra is a Jordan superalgebra.

Recall that a Jordan superalgebra is called {\it special} if it embeds into the Jordan
superalgebra associated with an associative superalgebra.
Namely, let $(\mathcal{A}=\mathcal{A}_0\oplus \mathcal{A}_1, \bullet)$ be an associative superalgebra. Then $\mathcal{A}^{(+)}=(\mathcal{A}, \cdot)$ is a Jordan superalgebra, where the new product is given for $x, y \in \mathcal{A}_0 \cup \mathcal{A}_1$ by 
$x\cdot y = \frac{1}{2}(x \bullet y + (-1)^{ab}y\bullet x)$. Now, a Jordan superalgebra is called special if it can be embedded in a superalgebra $\mathcal{A}^{(+)}$, for some associative superalgebra $\mathcal{A}$. Otherwise, it is called {\it exceptional}.
We obtain the following result.

\begin{proposition}\label{specialkantorvect}
Let $(\mathcal A,\cdot)$ be an associative commutative algebra over a field of characteristic $p\neq 2$, let $D$ be a derivation of $\mathcal A$, and let $q\in \mathcal A$. Consider on $\mathcal A$ the multiplications
$$
x\circ_q y=xqy \qquad \textrm{and} \qquad [x,y]=xD(y)-D(x)y \qquad \textrm{for any } x,y\in \mathcal{A}.
$$
Then the Kantor double of the algebra $(\mathcal A, \circ_q, [\cdot, \cdot])$ is a special Jordan superalgebra.
\end{proposition}
\begin{proof}
    It suffices to notice that the algebra $(\mathcal A, \circ_q, *)$ where $x * y = xD(y)$ is a generalized Novikov-Poisson algebra in the sense of Zhelyabin-Zakharov \cite{zz15}. Then by \cite[Theorem 3]{zz15}, the Kantor double of $(\mathcal A, \circ_q, [\cdot, \cdot])$ is special.
\end{proof}

As an immediate consequence, any mutation of a unital transposed Poisson algebra has special Kantor double.
In particular, we obtain the following consequence for the family $\mathcal{W}_n(q)$.

\begin{corollary}
For every $q\in \mathcal W(1;n)$, the Jordan superalgebra $\mathfrak J(\mathcal{W}_n(q))$ is special.
\end{corollary}

\subsection{The algebra $\mathcal{W}_n(q)$ as a quasi-Poisson algebra}

Quasi-Poisson algebras were introduced by Billig in \cite{billig} as a source of superconformal algebras. A quasi-Poisson algebra is a unital associative commutative algebra endowed with a Lie bracket and a linear operator $P$ satisfying suitable compatibility identities. The case $P=\mathrm{Id}$ is particularly relevant for us: then the quasi-Poisson identity is equivalent to the transposed Leibniz rule. Thus, transposed Poisson algebras may be viewed as quasi-Poisson algebras with $P=\mathrm{Id}$.

Billig proved that quasi-Poisson algebras give rise to Lie superalgebras of the form $L(A)$ using the construction in \cite[Theorem 6.8]{billig}, and that in the finite-dimensional case $L(A)$ is a superconformal algebra whenever the operator
$
Q=\operatorname{ad}(1)
$ ($Q_1$ in the notation of our paper)
is diagonalizable, see \cite[Theorem 6.10]{billig}. In particular, for the algebra $\mathcal{W}_n(q)$ with $q$ invertible, the associated distinguished operator is
$Q=\operatorname{ad}(q^{-1})$. 
Its behavior depends on the choice of $q$. For example, for the model algebra $\mathcal W_n(e_{-1})$, we have $1=e_{-1}$, and hence $Q=\operatorname{ad}(e_{-1})$ is nilpotent. On the other hand, diagonalizable cases also occur. For example, if $n=1$ and $q=(e_{-1}+e_0)^{-1}$, 
then $1=e_{-1}+e_0$, and for $-1\leq j\leq p-2$ we have
$$
Q(e_j)=[e_{-1}+e_0,e_j]=e_{j-1}+j e_j.
$$
Thus, in the basis $e_{-1},e_0,\dots,e_{p-2}$, the matrix of $Q$ is triangular with diagonal $(-1,0,1,\dots,p-2),$ 
which are pairwise distinct. Hence, $Q$ is diagonalizable.

Therefore, the family $\mathcal W_n(q)$ exhibits both nilpotent and diagonalizable distinguished operators $Q$. This makes it a natural modular source for the cited superalgebra construction and suggests that the algebras $\mathcal W_n(q)$ may be relevant in the search for modular analogues of superconformal Lie superalgebras. Moreover, it would be interesting to investigate whether the algebras $\mathcal W_n(q)$ admit extensions to quasi-Poisson superalgebras with nonzero odd part, and to determine the corresponding Lie superalgebras arising from Billig's construction. In this direction, the irreducible representations determined in the present paper provide a starting point.

\section*{Statements and Declarations}

\noindent{\bf Competing Interests}.
The author has no competing interests to declare.

\end{document}